\documentclass[10pt]{amsart}
\usepackage{amsthm}
\usepackage{amstext}
\usepackage{amssymb}
\usepackage{mathrsfs}
\usepackage{url}

\makeatletter
\numberwithin{equation}{section}
\numberwithin{figure}{section}
\theoremstyle{plain}
\newtheorem{thm}{Theorem}[section]
\theoremstyle{definition}
\newtheorem{definition}[thm]{Definition}

\newtheorem*{claim}{Claim}
\theoremstyle{plain}
\newtheorem{prop}[thm]{Proposition}
\theoremstyle{remark}
\newtheorem{rem}[thm]{Remark}
\theoremstyle{plain}
\newtheorem{cor}[thm]{Corollary}
\theoremstyle{plain}
\newtheorem{lem}[thm]{Lemma}
\theoremstyle{remark}
\newtheorem*{rem*}{Remark}

\newcommand{\abs}[1]{\left\vert#1\right\vert}
\newcommand{\set}[1]{\left\{#1\right\}}
\newcommand{\Real}{\mathbb{R}}
\newcommand{\df}{\mathrm{d}}

\newcommand{\ad}{\mathrm{ad}}

\newcommand{\tr}{\mathrm{tr}}
\newcommand{\Der}{\mathrm{Der}}
\newcommand{\Ric}{\mathrm{Ric}}
\newcommand{\Hess}{\mathrm{Hess}}
\newcommand{\II}{\mathrm{II}}
\newcommand{\scal}{\mathrm{scal}}

\newcommand{\frakn}{\mathfrak{n}}
\newcommand{\fraka}{\mathfrak{a}}
\newcommand{\frakk}{\mathfrak{k}}
\newcommand{\frakh}{\mathfrak{h}}
\newcommand{\frakg}{\mathfrak{g}}
\newcommand{\frakp}{\mathfrak{p}}
\newcommand{\frakq}{\mathfrak{q}}

\newcommand{\frake}{\mathfrak{e}}

\newcommand{\half}{\frac{1}{2}}
\newcommand{\sph}{\mathbb{S}}
\newcommand{\To}{\rightarrow}

\newcommand{\scp}[1]{\left\langle #1\right\rangle}

\begin{document}

\title[Homogeneous WPE and Ricci soliton metrics]{Warped product Einstein metrics on homogeneous spaces and homogeneous Ricci solitons}

\author{Chenxu He}
\address{
Chenxu He\\
Department of Mathematics \\
University of Oklahoma \\
Norman, OK 73019-3103}
\email{he.chenxu@ou.edu}
\urladdr{\url{http://sites.google.com/site/hechenxu/}}

\author{Peter Petersen}
\address{
Peter Petersen\\
520 Portola Plaza\\
Dept of Math UCLA\\
Los Angeles, CA 90095 \\}
\email{petersen@math.ucla.edu}
\urladdr{\url{http://www.math.ucla.edu/~petersen}}
\thanks{The second author was supported in part by NSF-DMS grant 1006677}

\author{William Wylie}
\address{
William Wylie\\
215 Carnegie Building\\
Dept. of Math, Syracuse University\\
Syracuse, NY, 13244.}
\email{wwylie@syr.edu}
\urladdr{https://wwylie.expressions.syr.edu/}
\thanks{The third author was supported in part by NSF-DMS grant 0905527}


\maketitle

\begin{abstract}
In this paper we consider connections between Ricci solitons and Einstein metrics on homogeneous spaces.  We  show that a semi-algebraic Ricci soliton admits an Einstein one-dimensional extension if  the soliton derivation can be chosen to be normal. Using our previous work on warped product  Einstein metrics, we show that every normal semi-algebraic Ricci soliton also admits a $k$-dimensional Einstein extension for any $k\geq 2$.  We also prove converse theorems for these constructions and  some geometric and topological structure results for homogeneous warped product Einstein metrics.    In the appendix we give an alternative approach to semi-algebraic Ricci solitons which naturally leads to a definition of semi-algebraic Ricci solitons in the non-homogeneous setting.
\end{abstract}

\section{Introduction}

A Riemannian manifold $(M, g)$ is called an \emph{Einstein manifold} if its Ricci curvature satisfies  $\Ric = \lambda g$ for some constant $\lambda \in \Real$. In this paper we are interested in non-compact homogeneous Einstein manifolds.   There are also many existence and nonexistence results if the manifold is compact, see for example \cite{WangZiller}.  From the classical Bonnet-Myers Theorem, an Einstein manifold is compact if $\lambda>0$. If $\lambda=0$, a homogeneous Ricci flat space is necessary flat, see \cite{AlekKim}. So for non-compact homogeneous Einstein manifolds one can assume that $\lambda< 0$.  

All known examples of non-compact, nonflat homogeneous Einstein manifolds are isometric to Einstein solvmanifolds. A solvmanifold $(G, g)$ is a simply-connected solvable Lie group $G$ endowed with a left invariant metric $g$. It has been conjectured by D. V. Alekseevskii that any noncompact, nonflat, homogeneous Einstein space $M$ has maximal compact isotropy subgroups, see \cite{Besse} and \cite{Alekseevskii}. If $G$ is a linear group that acts transitively on $M$,  this implies that $M$ is a solvmanifold or is diffeomorphic to a Euclidean space, see \cite[Section 2]{Heber}. Einstein solvmanifolds have been intensively investigated in \cite{Heber} and \cite{LauretStandard}.

A natural generalization of an Einstein manifold is a  Ricci soliton, i.e., a metric that satisfies the equation
\begin{equation}\label{eqn:RSLieX}
\Ric = \lambda g + \half \mathscr{L}_X g
\end{equation}
where $X \in \mathfrak{X}(M)$ is a smooth vector field and $\mathscr{L}_X$ is the Lie derivative. A trivial example of a Ricci soliton is an Einstein metric with $X$ a Killing vector field.  A Ricci soliton is called non-trivial if $X$ is not a Killing vector field. Under the Ricci flow, a Ricci soliton metric evolves via diffeomorphism and scaling. Besides the important role in the singularity analysis of Ricci flows, the geometry of Ricci solitons shares some common features with Einstein manifolds. 

A Ricci soliton is called a gradient Ricci soliton if $X$ is a gradient vector field. From the work of Ivey, Naber, Perelman and Petersen-Wylie, a nontrivial, nonflat homogeneous Ricci soliton must be noncompact, expanding ($\lambda<0$)  and of non-gradient type.  In fact, all known examples are isometric to a left-invariant metric $g$ on a simply connected solvable Lie group $G$, which when identified with an inner product on the Lie algebra $\mathfrak{g}$ of $G$ satisfies
\begin{equation}
\Ric = \lambda I + D \label{eqn:AlgSol}
\end{equation}
for some $\lambda \in \Real$ and $D\in \Der(\mathfrak{g})$ a symmetric derivation. On the other hand, any left invariant metric which satisfies the above equation is automatically a Ricci soliton and the diffeomorphisms which are generated by $D$ are automorphisms on the Lie algebra $\mathfrak{g}$. A generalized version of the Alekseevskii conjecture claims that these exhaust all examples of nontrivial, nonflat homogeneous Ricci solitons.

Recall that for a Lie group $H$, a Riemannian manifold, $(N,g)$ is called $H$-homogeneous if $H$ acts transitively and by isometries on $N$. The concept of a semi-algebraic Ricci soliton was  introduced  recently by M. Jablonski in \cite{Jablonski} via Ricci flows on homogeneous spaces. Roughly speaking,  a Ricci soliton on a $H$-homogeneous is  \emph{semi-algebraic with respect to $G$ } if the Ricci flow with initial metric $g$  flows by scaling and automorphisms which preserve the $H$-homogeneous structure (see Definition \ref{def:Jabl} below).   Jablonski proves several interesting results, including that every Ricci soliton on a homogeneous space is semi-algebraic with respect to its isometry group.  On the other hand, when $H$ is a proper subgroup of the isometry group, a  $H$-homogeneous Ricci soliton may not be semi-algebraic with respect to $H$ (see Example 1.3 of \cite{Jablonski}).

 In this paper we introduce the notion of a \emph{normal} semi-algebraic Ricci soliton.  Roughly speaking, a  semi-algebraic soliton on a $G$-homogeneous space  is normal if   a certain derivation, $D$,  on the Lie algebra  commutes with its adjoint when projected to the tangent space, see Section 2.1 for details.  It is an easy consequence of the definition that all algebraic Ricci solitons are normal.   Our first result shows that being a normal semi-algebraic Ricci soliton is the condition which allows us to construct a one dimensional extension which is Einstein.

\begin{thm}\label{thm:Einstein1extsoliton}
A non-flat, non-trivial  normal semi-algebraic Ricci soliton on a homogeneous space admits an  Einstein  one-dimensional extension.
\end{thm}

\begin{rem}
This theorem extends part of the work of J. Lauret on nilpotent groups which states that an algebraic Ricci soliton on a nilpotent group admits an Einstein one-dimensional extension  \cite{LauretNil}. His argument relies on the special curvature properties of nilpotent Lie groups. It is unknown  whether there are semi-algebraic Ricci solitons which are not algebraic.
\end{rem}

\begin{rem}
Our construction shows that if there is a normal semi-algebraic or algebraic Ricci soliton that is not isometric to a simply-connected solvable Lie group, then there is also a homogeneous Einstein manifold which is not isometric to an Einstein solvmanifold, see Theorems \ref{thm:EinsteinExtDnormal}, \ref{thm:LieStruc} and their remarks. This would give a counter-example to the Alekseevskii conjecture.  This result was also obtained in the case of algebraic Ricci solitons in a recent preprint of R. Lafuente and J. Lauret in \cite{LafuenteLauret} using different methods. 
\end{rem}

Our next result gives another connection between semi-algebraic Ricci solitons and homogeneous Einstein manifolds. This new connection is obtained by studying a special construction of Einstein metrics as  warped product metrics. For constants $\lambda \in \Real$ and $m\ne 0$ the space of all solutions to the \emph{$(\lambda, n+m)$-Einstein equation} on a Riemannian manifold $(M^n, g)$ is the following function space
\begin{equation}\label{eqn:Wspacewpe}
W(M, g) = W_{\lambda, n+m}(M, g) = \set{w\in C^{\infty}(M): \Hess w = \frac{w}{m}\left(\Ric - \lambda g\right)}
\end{equation}
A nonzero constant function is in $W(M, g)$ if and only if $(M, g)$ is a $\lambda$-Einstein manifold. When $m\geq 2$ is a positive integer, then $W(M, g)$ contains a positive function $w$ if and only if the product $E = M \times F^m$ with metric $g_E = g + w^2 g_F$ is a $\lambda$-Einstein manifold where the fiber $(F, g_F)$ is an appropriate space form. We call such a manifold $(M, g)$ a \emph{$(\lambda,n+m)$-Einstein manifold}. In this case, we also require $w = 0$ on $\partial M$ if it is non-empty. A $(\lambda, n+m)$-Einstein manifold is called \emph{non-trivial} if the warping function is not a constant. 

\begin{thm}\label{thm:WPE1extensionintro}
Let $m > 0$ be an integer and $\lambda < 0$ be a constant. A non-flat, non-trivial  normal semi-algebraic Ricci soliton admits a homogeneous $(\lambda, n+m)$-Einstein one-dimensional extension. 
\end{thm}

Theorems \ref{thm:Einstein1extsoliton} and \ref{thm:WPE1extensionintro} imply the following
\begin{cor}\label{cor:SARiccisiltonDnormal}
Let $m \geq 0$ be an integer. A non-flat, non-trivial normal semi-algebraic Ricci soliton on a homogeneous space $N^{n-1}$ admits a homogeneous Einstein extension $E^{n+m}$.
\end{cor}

\begin{rem}
Another interesting consequence of our construction is that every normal semi-algebraic Ricci soliton  can be isometrically embedded into a homogeneous Einstein manifold with an arbitrary codimension. 
\end{rem}

\begin{rem}
One of the main tools in the proof of Theorems \ref{thm:Einstein1extsoliton} and \ref{thm:WPE1extensionintro} is a simple construction, called a \emph{one-dimensional extension} of a homogeneous space, see the definition in section 2. It is a natural generalization of the semi-direct product of a Lie group with the real line $\Real$, i.e., an abelian extension with the real line. The Ricci curvature of the extension enjoys very nice properties when the original homogeneous space has a normal semi-algebraic Ricci soliton structure, see Lemma \ref{lem:RicciGKHK}.
\end{rem}

We also prove  that a converse of both Theorems \ref{thm:Einstein1extsoliton} and \ref{thm:WPE1extensionintro} hold, i.e., if the one-dimensional extension of a semi-algebraic Ricci soliton by $D$  is Einstein or $(\lambda, n+m)$-Einstein, then $D$ is normal. See Theorems \ref{thm:EinsteinExtDnormal} and \ref{thm:WPEextDnormal}.  In the case of $(\lambda, n+m)$-Einstein metrics we also prove the following characterization of the spaces in Theorem \ref{thm:WPE1extensionintro}. 

\begin{thm} \label{thm:WPEradial}
Let $(M, g)$ be a non-trivial  homogeneous   $(\lambda, n+m)$-Einstein metric. Then $M$ is a one-dimensional extension of a normal semi-algebraic Ricci soliton if and only if $\nabla_{\nabla w} \mathrm{Ric}=0$.
\end{thm}

Combining this result with Theorem \ref{thm:Einstein1extsoliton} gives us the following corollary. 

\begin{cor} \label{cor:WPEradial}
A homogeneous warped product Einstein metric, $(E, g_E)$ of the form  $g_E = g_M + w^2 g_F$  with $\nabla _{\nabla w} \mathrm{Ric} =0$ on $M$ is diffeomorphic to a product of Einstein metrics.
\end{cor}

\begin{rem}
We do not know of examples of homogeneous   $(\lambda, n+m)$-Einstein metrics which do not satisfy $\nabla _{\nabla w} \mathrm{Ric} =0$. In Theorem \ref{thm:structureWPEextension} we also give a structure theorem for the case where  $\nabla _{\nabla w} \mathrm{Ric} \neq 0$.   $M$ must still be a one-dimensional extension of a space $N$, but $N$ satisfies a slightly different equation than the semi-algebraic soliton equation. We do not know if there are examples that satisfy this equation but are not Ricci solitons. 
\end{rem}

\smallskip

The paper is organized as follows. In section 2  we review the definition of semi-algebraic Ricci solitons, define one-dimensional extensions of homogeneous spaces, and recall some useful facts of $(\lambda, n+m)$-Einstein manifolds. In section 3 we study two special cases of one-dimensional extensions: when the extension is Einstein and when the extension is  $(\lambda, n+m)$-Einstein. The study of the first case gives a proof of Theorem \ref{thm:Einstein1extsoliton}. In section 4 we apply the results on $(\lambda, n+m)$-Einstein metrics with symmetries in our earlier paper \cite{HPWuniqueness} to study homogeneous $(\lambda, n+m)$-Einstein manifolds. In section 5 we characterize the structure of homogeneous $(\lambda, n+m)$-Einstein manifolds and prove Theorems \ref{thm:WPE1extensionintro} and \ref{thm:WPEradial}. In section 6 we specialize our study of general homogeneous spaces to Lie groups with left invariant metrics.   In the appendix, we also give an alternative approach to semi-algebraic Ricci solitons in terms of algebras of vector fields and propose a definition of semi-algebraic Ricci solitons on non-homogeneous spaces. 

\smallskip

\textbf{Acknowledgements.} Part of the work was done when the first author was at Lehigh University and he is very grateful to the institute for their hospitality.   

\medskip
\section{Preliminaries}

This section is separated into three subsections. In the first subsection we recall the definiton of  semi-algebraic Ricci solitons. In the second subsection we consider the useful construction of a \emph{one-dimensional extension} of a homogeneous space and study how its curvature relates to those on the original manifold. In the third subsection we collect a few relevant facts about $(\lambda, n+m)$-Einstein manifolds from \cite{HPWLcf,HPWuniqueness}.

\subsection{Semi-algebraic Ricci solitons} 
We recall the definition of a homogeneous semi-algebraic soliton given in \cite{Jablonski}. First we fix some notation.     Let $H$ be a Lie group and  $(M = H/K, g)$ be an $H$-homogeneous space.  Let $K$ is the isotropy subgroup at a fixed point $x\in M$ and $\frakh, \frakk$ be the Lie algebras of $H$ and $K$ respectively.   Let $\Phi_t \in Aut(H)$ be a family of automorphisms of $H$ such that $\Phi_t(K) = K$;  $\Phi_t$ gives rise to a well defined diffeomorphism $\phi_t$ of $H/K$ defined by 
\[ \phi_t(hK) = \Phi_t(h) K \qquad h \in H. \]

\begin{definition} \label{def:Jabl} \cite[Definition 1.4]{Jablonski}
$(H/K,g)$ is a semi-algebraic Ricci soliton with respect to $H$ if there exists a family of automorphisms $\Phi_t \in Aut(H)$ such that $\Phi_t(K) = K$ and 
\[ g_t = c(t) \phi^*_t(g) \]
is a solution to the Ricci flow
\[ \frac{\partial}{\partial t} g = - 2 \mathrm{Ric}_g \] on $H/K$ with $g_0=g$.
\end{definition}

Fix an $\mathrm{Ad}(K)$-invariant decomposition $\mathfrak{h} = \mathfrak{p} \oplus \mathfrak{k}$  and let $\mathrm{pr}: \frakh \rightarrow \frakp$ be  the orthogonal projection.  $\frakp$  is  then naturally identified with $T_xM$. Jablonski also proves the following proposition about semi-algebraic solitons. 
   
\begin{prop} \label{prop:Jabl}  \cite[Proposition 2.3]{Jablonski}
If $(H/K,g)$ is a semi-algebraic Ricci soliton with respect to a Lie Algebra $H$ then there exists  a derivation $D \in \Der(\frakh)$ such that
\[
\Ric = \lambda I + \half \left(\mathrm{pr}\circ D + (\mathrm{pr}\circ D)^*\right).
\] 
Here   $^* $ denotes the adjoint with respect to the metric $g$  on $\frakh$.  Moreover, we may assume that $D|_{\frakk} = 0$. 
\end{prop}

The condition which will become important in constructing extensions is that the derivation $D$ be a normal operator, at least when projected to $\frakp$.  In particular, we give the following definition. 

\begin{definition} A semi-algebraic Ricci soliton is \emph{normal} if the map $\mathrm{pr} \circ D \circ \mathrm{pr} : \frakg \rightarrow \frakg$ is a normal operator. \end{definition}

\begin{rem} A Ricci soliton on a $H$ homogeneous space is called called \emph{algebraic} if 
\[ \mathrm{Ric} = \lambda I + \mathrm{pr} \circ D \]
for some $D \in \Der(\frakh)$.  Since the Ricci tensor is a symmetric operator, algebraic solitons are always normal.  
\end{rem}

Recall that, if we consider the symmetric and anti-symmetric parts of $D$
\begin{eqnarray}
S &=& \frac{1}{2}(\mathrm{pr} \circ D \circ \mathrm{pr}  + (\mathrm{pr} \circ D \circ \mathrm{pr})^*) \label{eqn:symmD}\\
A &=&  \frac{1}{2}(\mathrm{pr} \circ D \circ \mathrm{pr} - (\mathrm{pr} \circ D \circ \mathrm{pr})^*), \label{eqn:antisymmD}
\end{eqnarray}
the operator $\mathrm{pr}  \circ D \circ \mathrm{pr}$ will be normal if and only if $S$ and $A$ commute, $[S,A]=0$. We will find $[S,A]$ to be an important  term in the calculation of Ricci curvatures of extensions in the next subsection. 

\subsection{One-dimensional extension of homogeneous spaces}
We recall some general facts about extensions of Lie groups and Lie algebras, i.e., semi-direct products. Let $H$ be a Lie group and let $(N,h)$ be an $H$-homogeneous space. By passing to its universal cover if necessary, we may assume that $H$ is simply-connected. We use the same notation for $K$, $\frakh$, $\frakk$, $\frakp$, $\mathrm{pr}$ as in the previous subsection. 
To construct an extension we fix a constant $\alpha\in \Real$ and a derivation of the Lie algebra $D \in \mathrm{Der}(\mathfrak{h})$  which preserves $K$ and consider  the new Lie algebra
\[ \mathfrak{g} = \mathfrak{h} \oplus \mathbb{R} \xi \]
on which the Lie bracket operation is given by
\[ \mathrm{ad}_{\xi}(X) =  \alpha D(X), \qquad \text{for all } X \in \mathfrak{h}.\]
Let $G$ be the simply-connected Lie group with Lie algebra $\mathfrak{g}$ that contains $H$ as a subgroup.  Since $\mathrm{ad}_{\xi}(X) \in \mathfrak{h}$ for any $X\in \mathfrak{h}$, $H$ is a codimension one normal subgroup of $G$ and $G$ is a semi-direct product $G= H \ltimes \mathbb{R}$.  Given the $\mathrm{Ad}(K)$-invariant decomposition $\mathfrak{h} = \mathfrak{p} \oplus \mathfrak{k}$, we have the corresponding  $\mathrm{Ad}(K)$-invariant decomposition $\mathfrak{g} = \mathfrak{q} \oplus \mathfrak{k}$, where $\mathfrak{q} = \mathfrak{p} \oplus \mathbb{R} \xi$, and we identify $G$-invariant metrics with the restriction of $\mathrm{Ad}(K)$-invariant inner products on $\mathfrak{g}$ to $\mathfrak{q}$.

This extension of Lie groups defines a natural extension of homogeneous spaces.
\begin{definition}
Let $(N, h)$ be an $H$-homogeneous space. For a constant $\alpha \in \Real$ and a derivation $D\in \Der(\mathfrak{h})$, the \emph{one-dimensional  extension} of $(N,h)$ is a $G$-homogeneous space $(M,g)$ with $M = G/K$ and
\begin{eqnarray*}
g |_{\mathfrak{p}} &=&  h, \\
g(\xi, X) &=& 0 \qquad \text{for all } X \in \mathfrak{p}, \\
g( \xi, \xi) &=& 1,
\end{eqnarray*}
where $G = H\ltimes \Real$ is the semi-direct product of $H$ and $\Real$ by $D$ and $\alpha$.
\end{definition}

\begin{rem}
Note that for the commutator series we have
\[
\frakh^1 = [\frakh, \frakh] \subset \frakg^1 = [\frakg, \frakg] \subset \frakh.
\]
It follows that $\frakh$ is a solvable Lie algebra if and only if $\frakg$ is solvable.
\end{rem}

In the following we  compute the curvatures of an extension and relate them back to the curvatures of $(N,h)$,  the derivation $D$,  and  the constant $\alpha$.  To do so we consider the codimension one submanifold $H/K \subset M = G/K$ which is the $H$ orbit at $x \in M$. The vector $\xi$ is a unit normal vector to $H/K$ at $x$. Using left translations of $H$ from $x$ we obtain a unit normal vector field to $H/K$  which is also denoted by $\xi$. The second fundamental form of $N \subset M$ is then given by
\[
\II_x (X, Y) = g(T(X), Y) = g(\nabla^M_X \xi, Y)
\]
where $T: T_x N \rightarrow T_x N$ is the shape operator. Recall that $S$ and $A$ denote  symmetric and anti-symmetric parts of $D$, restricted to $\frakp \simeq T_xM$, 
\begin{eqnarray}
S &=& \frac{1}{2}(\mathrm{pr} \circ D + \mathrm{pr} \circ D^*) \label{eqn:symmD}\\
A &=&  \frac{1}{2}(\mathrm{pr} \circ D - \mathrm{pr} \circ D^*), \label{eqn:antisymmD}
\end{eqnarray}

The next proposition relates the tensors $S$ and $A$ to the shape operator $T$. 

\begin{prop} \label{prop:1extensionST}
Let $(M,g)$ be the one-dimensional extension of $(N,h)$ with
\[ \mathrm{ad}_{\xi}(X) = \alpha D(X) \qquad \text{for all } X \in \mathfrak{h}.\]
Then
\begin{eqnarray*}
T &=&   -\alpha S \\
\nabla_{\xi} T &=&  \alpha^2 [S,A].
\end{eqnarray*}
\end{prop}

\begin{proof}
We start with the calculation on the Lie group $G$, endowed with a metric so that the quotient map $G \rightarrow G/K$ is a Riemannian submersion. Then we have
\begin{eqnarray*}
\alpha D(X) = \mathrm{ad}_{\xi}(X) = - \nabla^G_X \xi + \nabla^G_{\xi} X. 
\end{eqnarray*}
Since $\nabla^G_{\cdot} \xi$ is the shape operator of $H \subset G$ which is symmetric and $\nabla^G_{\xi} \cdot$ is skew-symmetric, we have
\begin{eqnarray*}
\nabla^G_{\xi} X  &=& \frac{\alpha}{2} \left( D - D^* \right) (X)\\
\nabla^G_{X} \xi &=& - \frac{\alpha}{2} \left( D + D^* \right) (X).
\end{eqnarray*}

Now choosing $X$ in $\mathfrak{p}$ is equivalent to $X$ being a basic horizontal field of the Riemannian submersion $G \rightarrow G/K$, see \cite[Chapter 3]{CheegerEbin}.  The unit vector $\xi$ is also basic and horizontal.  Therefore, we have
\begin{eqnarray*}
\nabla_{\xi} X   &=&  \mathrm{pr}\left( \nabla^G_{\xi} X\right) = \alpha A \\
T(X)  &=& \nabla_{X} \xi  =  \mathrm{pr}\left( \nabla^G_{X} \xi \right)=  - \alpha S.
\end{eqnarray*}
This also gives us
\begin{eqnarray*}
(\nabla_{\xi} T)(X) &=& \nabla_{\xi}(T(X)) - T\left( \nabla_{\xi} X \right) \\
&=& ((\alpha  A )\circ (- \alpha S) ) + (\alpha S) \circ(\alpha A) )(X)\\
&=& \alpha^2 [S,A](X)
\end{eqnarray*}
which finishes the proof.
\end{proof}

\begin{rem}  While the tensor $S$ is defined on a chosen tangent space, $T_xM$,  the formula $T=-\alpha S$, defines  $S$ geometrically over the entire manifold $M$ and allows us to define covariant derivatives  and the divergence of the tensor $S$.  Note that, since the unit normal vector $\xi$ is invariant under the group $H$, so is $T$ and thus $S$. 
\end{rem}

Combining these formulas with the Gauss, Codazzi, and radial curvature equations, see \cite{PetersenGTM}, we have the following formulas for the Ricci curvatures.

\begin{lem} \label{lem:RicciGKHK}
Let $(N,h)$ be an $H$-homogeneous space, $D$ a derivation of $\mathfrak{h}$ , and let $(M,g)$  the one-dimensional extension with
\[
\mathrm{ad}_{\xi}(X) = \alpha D(X) \qquad \text{for all }X \in \mathfrak{h}.
\]
Then the  Ricci tensor of $(M,g)$ is given by
\begin{eqnarray}
\mathrm{Ric}\left(\xi,\xi\right) & = & - \alpha^2 \tr(S^2)   \notag\\
\mathrm{Ric}\left(X,\xi\right) & = & - \alpha \mathrm{div}(S)  \label{eqn:RicS} \\
\mathrm{Ric}\left(X,X\right) & = & \mathrm{Ric}^{N}(X,X)-\left(\alpha^2 \mathrm{tr}S\right)h\left(S(X),X\right) - \alpha^2 h([S,A](X), X) \notag
\end{eqnarray}
\end{lem}

\begin{rem}
When $D$ is normal, we have $[S, A] = 0$. Thus the Ricci curvatures of $M$ do not depend on the skew-symmetric part $A$ and have a very simple form in terms of $S$ and Ricci curvatures of $N$. 
\end{rem}

\begin{proof}
The radial, Gauss, and Codazzi equations tell us the curvatures on $(M,g)$ have the following forms
\begin{eqnarray*}
R(X,\xi)\xi & = & -\left(\nabla_{\xi}T\right)\left(X\right)- T^{2}\left(X\right)\\
R(X,Y,Z,W) & = & R^{N}(X,Y,Z,W)-g(T(Y),Z)g\left(T(X),W\right)+g(T(X),Z)g\left(T(Y),W\right)\\
R(X,Y,Z,\xi) & = & -g\left(\left(\nabla_{X}T\right)\left(Y\right)-\left(\nabla_{Y}T\right)\left(X\right),Z\right).
\end{eqnarray*}
Let $\set{X_i}_{i=1}^{n-1}$ be an orthonormal basis of $\mathfrak{p}$, then the Ricci tensor satisfies
\begin{eqnarray*}
\mathrm{Ric}\left(\xi,\xi\right) & = & \sum R(X_{i},\xi,\xi,X_{i})\\
& = & -D_{\xi}(\mathrm{tr}T)-\mathrm{tr}\left(T^{2}\right) \\
\mathrm{Ric}\left(X,\xi\right) & = & \sum R(X_{i}, X, \xi, X_{i}) =  -g\left(\left(\nabla_{X}T\right)\left(X_{i}\right)-\left(\nabla_{X_{i}}T\right)\left(X\right),X_{i}\right)\\
& = & -D_{X}\left(\mathrm{tr}T\right)+\mathrm{div}T(X)
\end{eqnarray*}
and
\begin{eqnarray*}
\mathrm{Ric}\left(X,X\right) & = & \sum R(X_{i},X,X,X_{i})\\
& = & R(\xi,X,X,\xi)+\sum_{i<n}R(X_{i},X,X,X_{i})\\
& = & -g\left(\left(\nabla_{\xi}T\right)(X),X\right)-g\left(T^{2}(X),X\right) \\
& & +\mathrm{Ric}^{N}(X,X)-g(T(X_{i}),X_{i})g\left(T(X),X\right)+g(T(X),X_{i})g\left(T(X_{i}),X\right)\\
& = & \mathrm{Ric}^{N}(X,X)-\left(\mathrm{tr}T\right)g\left(T(X),X\right) - g\left(\left(\nabla_{\xi}T\right)(X),X\right).
\end{eqnarray*}
Note that $T$ is invariant under the isometries and so $\mathrm{tr}(T)$ is constant. Substituting  in $T=\alpha S$ and $\nabla_\xi T = -\alpha^2 [S,A]$ from Proposition \ref{prop:1extensionST} gives us
\begin{eqnarray}
\mathrm{Ric}\left(\xi,\xi\right) & = &- \alpha^2 \tr(S^2)   \notag\\
\mathrm{Ric}\left(X,\xi\right) & = & -\alpha \mathrm{div}(S)  \notag \\
\mathrm{Ric}\left(X,X\right) & = & \mathrm{Ric}^{N}(X,X)-\left(\alpha^2 \mathrm{tr}S\right)h\left(S(X),X\right) - \alpha^2 h([S,A](X), X) \notag
\end{eqnarray}
which finishes the proof.
\end{proof}

\subsection{$(\lambda, n+m)$-Einstein manifold with symmetries}
Recall that the space $W = W_{\lambda, n+m}(M, g)$ contains the solutions to the following $(\lambda, n+m)$-Einstein equation
\begin{equation}
\Hess w = \frac{w}{m}\left(\Ric - \lambda g\right).
\end{equation}
The space $W$ is clearly a vector space. Moreover, there is an associated quadratic form
\begin{equation}\label{eqn:muw}
\mu(w) = w \Delta w + (m-1)\abs{\nabla w}^2 + \lambda w^2.
\end{equation}
In the case when $(M, g)$ is a $(\lambda, n+m)$-Einstein manifold with the warping function $w$, $\mu(w)$ is the Ricci curvature of $(F, g_F)$ and the warped product metric $g + w^2 g_F$ is Einstein with Einstein constant $\lambda$. In a series of papers \cite{HPWLcf,HPWconstantscal,HPWuniqueness} the geometry of $(\lambda, n+m)$-Einstein manifolds has been intensively studied for its connection to comparison geometry of $m$-Bakry Emery tensors, gradient Ricci solitons, etc. In particular, in \cite{HPWconstantscal} we showed that there are non-trivial $(\lambda, 4+m)$-Einstein metrics on certain 4-dimensional solvable Lie groups with left invariant metrics. This gave the first examples of non-trivial homogeneous $(\lambda, n+m)$-Einstein manifolds. 

By studying the space $W$ with the quadratic form $\mu$ we obtained in \cite{HPWuniqueness} a few uniqueness theorems. A main result from \cite{HPWuniqueness} is the following structure theorem when $\dim W \geq 2$.

\begin{thm}[Theorem 2.6 in \cite{HPWuniqueness}]\label{thm:wpedecomposition}
Let ($M^n, g$) be a complete simply-connected Riemannian manifold with $\dim W_{\lambda, n+m}(M,g) = k+1$, then
\[
M = B^b \times_u F^k
\]
where
\begin{enumerate}
\item $B$ is a manifold, possibly with boundary, and $u$ is a nonnegative function in $W_{\lambda, b+(k+m)}(B, g_B)$ with $u^{-1}(0) = \partial B$,
\item the function $u$ spans $W_{\lambda, b+(k+m)}(B, g_B)$, and
\item the fiber $F^k$ is a space form with $\dim W_{\mu_B(u), k+m}(F,g_F) = k+1$, where $\mu_B$ denotes the quadratic form on $W_{\lambda, b+(k+m)}(B, g_B)$.
\end{enumerate}
Moreover,
\begin{enumerate} \setcounter{enumi}{3}
\item $W_{\lambda, n+m}(M,g) = \{ uv : v \in W_{\mu_B(u), k+m}(F, g_F) \}$.
\end{enumerate}
\end{thm}

The above theorem motivated the following
\begin{definition}
Let $(B^b, g_B)$ be a Riemannian manifold possibly with boundary  and let  $u$ a nonnegative function on $B$ with $u^{-1}(0) = \partial B$. Then $(B, g_B, u)$ is called a  \emph{$(\lambda,k+m)$-base manifold} if $\left(W_{\lambda, b+(k+m)}(B, g_B)\right)_D = \text{span}\set{u}$, where $W_D$ denotes the solutions satisfying Dirichlet boundary conditions. It is called an \emph{irreducible base manifold} if  $W_{\lambda, b+(k+m)}(B, g_B)= \text{span}\set{u}$ with no boundary conditions imposed.
\end{definition}

We showed that the converse of the above theorem also holds.

\begin{thm}[Theorem 2.10 in \cite{HPWuniqueness}] \label{thm:ExistenceMBF}
Given an irreducible $(\lambda,k+m)$-base manifold $(B, g_B, u)$ there is a complete metric of the form
\[ M = B^b \times_u F^k \]
such that  $\dim W_{\lambda, (b+k) +m} (M,g_M) = k+1$.
\end{thm}

\begin{rem}  If $\partial B = \emptyset$, $\mu_B(u)>0$, and $k=1$ there are two such metrics corresponding to the choice $F= \mathbb{R}$ or $F = \sph^1$.   Otherwise, the warped product over $B$ with  $\dim W_{\lambda, (b+k) +m} (M,g_M) = k+1$ is unique.    In this case, we call  $M$ the \emph{ $k$-dimensional elementary warped product extension} of $(B,g_B, u)$.
\end{rem}

In Section 5 of \cite{HPWuniqueness} we established the following relationship between isometries on $M$ and $B$. Note that any isometry $\sigma$ of $M$ has a natural action on $W(M, g)$ sending $w$ to $w\circ \sigma^{-1}$.

\begin{thm}[Theorem 5.7 in \cite{HPWuniqueness}]\label{thm:isometryMBF}
Let $M$ be a simply connected $(\lambda, n+m)$-Einstein manifold with  $\dim W_{\lambda, n+m}(M)=k+1>1$. Then the isometry group of $M$ consists of  maps  $h:M \rightarrow M$  of the form
\[
h = h_1 \times   h_2 \quad \text{with} \quad h_1: B \rightarrow B \quad h_2: F \rightarrow F,
\]
where $h_1 \in \mathrm{Iso}(B, g_B)$ and
\begin{enumerate}
\item If $\mu(u) \neq 0$ then $h_2 \in \mathrm{Iso}(F, g_F)$.
\item If $\mu(u) = 0$ then  $h_2$ is a $C$-homothety of $\mathbb{R}^k$  where $C = C_{h_1}$ is the constant so that $u \circ h_1^{-1} = C_{h_1} u$.  Namely,
\[
h_2(v) = b + C A(v) \quad \text{with} \quad b \in \mathbb{R}^k \text{ and } A \in \mathrm{O}(\mathbb{R}^k).
\]
\end{enumerate}
\end{thm}

\begin{rem}
Note that the isometry group $\mathrm{Iso}(M,g)$ contains the following subgroup 
\[
\mathrm{Iso}(B,g_B)_u \times \mathrm{Iso}(\Real^k),
\]
where $\mathrm{Iso}(B,g_B)_u = \set{\sigma \in \mathrm{Iso}(B,g_B) : u = u\circ \sigma^{-1}}$. This subgroup has codimension one unless $u$ is a constant function. 
\end{rem}

\medskip
\section{One-dimensional extension of homogeneous space}

In this section using Lemma \ref{lem:RicciGKHK} we obtain the conditions under which a homogeneous space admits an Einstein one-dimensional extension, see Theorem \ref{thm:EinsteinConstruction}, and then prove Theorem \ref{thm:Einstein1extsoliton} in the introduction, see Theorem \ref{thm:EinsteinExtDnormal}. In the second part, we show that it also admits a one-dimensional extension which is $(\lambda, n+m)$-Einstein with $\lambda < 0$, see Theorem \ref{thm:WPE1extension}.

\subsection{Einstein one-dimensional extension}
First we have
\begin{thm} \label{thm:EinsteinConstruction}
Suppose $(N,h)$ is an $H$-homogeneous space and $D$ is a derivation of $\mathfrak{h}$ such that the following conditions hold
\begin{enumerate}
\item $ \Ric^N = \lambda I + S + \dfrac{1}{\tr(S)} [S,A] $
\item $\mathrm{div}(S) = 0$
\item $\tr(S^2) = -\lambda \tr(S)$
\end{enumerate}
for some constant $\lambda <0$, then the one-dimensional extension of $(N,h)$  with $D$ and $\alpha^2 = 1/\tr(S)$ is $\lambda$-Einstein.
\end{thm}

\begin{proof}
From Lemma \ref{lem:RicciGKHK} we see that $\mathrm{div}(S) = 0$ is equivalent to $\mathrm{Ric}\left(X,\xi\right) = 0$. The condition $\tr(S^2) = -\lambda \tr(S)$ then tells us that
\[ \mathrm{Ric}\left(\xi,\xi\right)  = \lambda \alpha^2 \tr(S). \]
So if we choose $\alpha$ so that $\alpha^2 = \dfrac{1}{\tr(S)}$, then $\Ric(\xi, \xi)=\lambda$.  Then we have
\begin{eqnarray*}
\mathrm{Ric}\left(X,X\right) & = & \mathrm{Ric}^{N}(X,X)-h\left(S(X),X\right) - \frac{1}{\tr(S)}  h([S,A](X), X) \\
&=& \lambda g(X,X)
\end{eqnarray*}
from condition (1). This shows that the one-dimensional extension $(M,g)$ of $(N, h)$ is Einstein with Einstein constant $\lambda$.
\end{proof}

In the case when $(N, h)$ is a semi-algebraic Ricci soliton the theorem above gives us Theorem \ref{thm:Einstein1extsoliton} in the introduction.

\begin{thm}\label{thm:EinsteinExtDnormal}
A non-flat, non-trivial  semi-algebraic Ricci soliton  on a homogeneous space
admits an  Einstein  one-dimensional extension with
\[
\mathrm{ad}_{\xi}(X) =  \alpha D(X), \qquad \text{for all }X \in \mathfrak{h}
\]
if and only if $D$ is normal.
\end{thm}

\begin{proof}
Let $(N^{n-1},h)$ be a semi-algebraic Ricci soliton
\[
\mathrm{Ric} = \lambda I +S
\]
for some constant $\lambda$. By exponentiating, we can identify $S$ with the Lie derivative $\frac{1}{2} \mathscr{L}_Y h $ for some vector field $Y \in \mathfrak{X}(N)$ and get a Ricci soliton structure on $N$.  This gives us
\[
\mathrm{div}( \mathscr{L}_Y h) =  2 \mathrm{div}(\mathrm{Ric}) = d \mathrm{scal} = 0
\]
and so $\mathrm{div}(S) = 0$. Recall the formula for the Laplacian of the scalar curvature of a Ricci soliton
\[
\Delta(\mathrm{scal}) - D_Y \mathrm{scal} = \lambda \mathrm{scal} - |\mathrm{Ric}|^2.
\]
Since  the scalar curvature is constant, we have
\[
\abs{\mathrm{Ric}}^2 = \lambda \mathrm{scal}
\]
which is equivalent to $\mathrm{tr}(S^2)  =  -\lambda \mathrm{tr}(S)$.

If $D$ is normal, then $[S,A]=0$ and then from the previous Theorem \ref{thm:EinsteinConstruction} there is a one-dimensional extension of $(N, h)$ which is Einstein.

On the other hand, if we have a non-trivial semi-algebraic Ricci soliton and a one-dimensional extension which is $\lambda$-Einstein, then the equation $\mathrm{tr}(S^2)  =  -\lambda \mathrm{tr}(S)$ and Ricci curvature $\Ric^M(\xi, \xi) = \alpha^2 \tr(S^2) = \lambda$ imply that $\alpha^2 \tr(S) = 1$.  Plugging this into the equation for $\mathrm{Ric}^M(X,X)$ in Lemma \ref{lem:RicciGKHK} then implies that $[S,A]=0$, i.e., $D$ is normal.

To finish the proof we show that if the extension has  Einstein constant  $c \neq \lambda$ then $N$ is flat.   Now we must choose $\alpha$ so that  $\alpha^2 = \dfrac{c}{\lambda \tr(S)}$, then the third equation in Lemma \ref{lem:RicciGKHK} tells us that on $N$ we have
\begin{eqnarray*}
\alpha^2 [S,A] + \left(\frac{c}{\lambda}  - 1\right) S + (c - \lambda) I = 0.
\end{eqnarray*}
Tracing this equation yields
\begin{eqnarray*}
\left(\frac{c}{\lambda}  - 1\right) \tr(S) + (c - \lambda) (n-1) = 0.
\end{eqnarray*}
Since $c \neq \lambda$ this is equivalent to
\begin{eqnarray*}
\tr(S) + \lambda(n-1) = 0.
\end{eqnarray*}
But this implies we have $\mathrm{scal}=0$ which implies $|\mathrm{Ric}|^2 =0$ on $N$. So $N$ is a flat space as a homogeneous Ricci flat space is flat, see for example, \cite[Theorem 7.61]{Besse}.
\end{proof}

\smallskip
\subsection{$(\lambda, n+m)$-Einstein one-dimensional extension}

We consider the same set-up as in the previous subsection. Namely $(N^{n-1}, h)$ is an $H$-homogeneous space and  $(M^n,g)$ is a one-dimensional extension of $N$ by a derivation $D\in \Der(\frakh)$ and a constant $\alpha$.  We let $r$ be the signed distance function on $M$ to the hypersurface $N$.  For reasons made clear below, we let
\[ w(r) = e^{L r}.  \]
for some constant $L$ which is also determined later.

\begin{thm}\label{thm:WPE1extension}
Let $m>0$ be an integer.  Suppose $(N^{n-1},h)$ is an $H$-homogeneous space with a derivation $D \in \Der(\mathfrak{h})$ which satisfies the following conditions
\begin{enumerate}
\item $\Ric^N = \lambda I + S + \dfrac{1}{\tr(S)-\lambda m} [S,A] $
\item $\mathrm{div}(S) = 0$
\item $\tr(S^2) = -\lambda \tr(S)$
\end{enumerate}
for some constant $\lambda <0$. Then the one-dimensional extension of $(N, h)$ with $D$ and $\alpha^2 = 1/(\tr(S) - \lambda m)$ is a $(\lambda, n+m)$-Einstein manifold with warping function 
\[
w(r) = e^{Lr} \quad \text{where} \quad L = \lambda \alpha.
\]
\end{thm}

\begin{proof}
Setting $w=e^{Lr}$ we have
\begin{eqnarray*}
\mathrm{Hess} w = L^2 w dr\otimes dr + L w g(T(\cdot),\cdot)
\end{eqnarray*}
where $dr\otimes dr(X, Y) = X(r) Y(r)$ and $T(X) = \nabla_X \nabla r$ for any $X, Y \in T_xM$. Note that $\xi = \nabla r$. Applying the assumption on $N$ along with  Lemma \ref{lem:RicciGKHK} we obtain the following equations
\begin{eqnarray*}
\left( \mathrm{Ric} - \frac{m}{w} \mathrm{Hess} w\right)(\xi, \xi) &=& \lambda \alpha^2 \tr(S)  - mL^2 \\
\left( \mathrm{Ric} - \frac{m}{w} \mathrm{Hess} w\right)( X, X) &=& \lambda h(X,X)+(1 -\alpha^{2}\left(\mathrm{tr}S\right) + mL \alpha )h\left(S(X),X\right) \\
&& + \left(\frac{1}{\tr(S)-\lambda m}- \alpha^2\right) h([S,A](X), X)
\end{eqnarray*}
We wish to show there is a solution of $\alpha$ and $L$ such that $ \mathrm{Ric} - \dfrac{m}{w} \mathrm{Hess} w = \lambda g$.  From the first equation above we can see that a necessary condition is that
\[
\alpha^2 \tr(S) = 1 + \frac{m}{\lambda}L^2.
\]
Plugging this condition into the trace of the second equation above we obtain,
\[
mL\left(\alpha - \frac{L}{\lambda}\right) = 0
\]
which indicates we should either choose $L=0$ or $L= \lambda \alpha$. The case $L=0$ corresponds to the Einstein extension discussed in Theorem \ref{thm:EinsteinConstruction}. In the other case plugging $L= \lambda \alpha$ back into the first equation gives
\[
\alpha^2 = \frac{1}{\tr(S) - \lambda m}
\]
This choice of $\alpha$ also makes the $[S,A]$ term vanish.  Note that since $\tr(S)>0$ and $\lambda<0$ the quantity on the right is positive, so there is a solution to this equation.
\end{proof}
We note that, in the special case where $N$ is a normal semi-algebraic Ricci soliton this gives us the following. 

\begin{cor} \label{cor:SolitonQuasiEinstein}
Let $(N,h)$ be a non-flat  H-homogeneous normal semi-algebraic Ricci soliton.  For every $m>0$ there is a one-dimensional extension of $(N,h)$, $(M, g_m)$ such that $(M,g_m)$ is a $(\lambda, n+m)$-Einstein manifold.
\end{cor}

\begin{proof}
Let $(N^{n-1}, h)$ be a semi-algebraic Ricci soliton
\[
\Ric = \lambda I + S
\]
for some constant $\lambda < 0$. As in the proof of Theorem \ref{thm:EinsteinExtDnormal}, we have 
\[
\mathrm{div} S = 0,\quad \tr(S^2) = - \lambda \tr(S) \quad \text{and}\quad \abs{\Ric}^2 = \lambda \scal.
\]
If $D$ is normal, then $[S,A] = 0$ and from Theorem \ref{thm:WPE1extension} there is a one-dimensional extension of $(N, h)$ which is $(\lambda, n+m)$-Einstein. 
\end{proof}

\begin{rem}
The examples in Theorems \ref{thm:EinsteinConstruction} and \ref{thm:WPE1extension} illustrate that there can  be different base metrics on a given topological manifold which produce different homogeneous warped product Einstein metrics. More precisely, let $(N^{n-1},h)$ be a semi-algebraic Ricci soliton with a normal derivation $D$ and  $M^n = N \times \mathbb{R}$. Then on the product manifold
\[
E^{n+m} = M \times \mathbb{R}^m
\]
there are two homogeneous, warped product Einstein metrics.  One is the product metric $g_0 + g_{H^m}$ where $g_0$ is the Einstein metric on $M$ constructed in Theorem \ref{thm:EinsteinConstruction} and $g_{H^m}$ is the hyperbolic space with Ricci curvature $\lambda$.  The other is the metric $g_m + e^{Lr} g_{\mathbb{R}^m}$ where $g_m$ is the metric constructed in Theorem \ref{thm:WPE1extension}, and $g_{\mathbb{R}^m}$ is the Euclidean metric.  We will verify that the metric $g_m + e^{Lr} g_{\mathbb{R}^m}$ is homogeneous in Theorem \ref{thm:homogeneousE}
\end{rem}

\begin{rem}
The family of metrics $g_m$ only differ by the choice of the structure constant given by $\alpha$.  As $m \rightarrow \infty$, both $\alpha$ and  $L$ approach $0$ and thus the metrics $g_m$ on $M$ are converging to the product metric $g = h+ dt^2$ on $N \times \mathbb{R}$. Moreover,  for any $X\in TN$,  we have
\begin{eqnarray*}
\frac{m}{w} \mathrm{Hess} w(X, X) &=& - mL\alpha h(S(X),X))\\
&=& -\lambda m \alpha^2 h(S(X),X)\\
&=&  - \frac{\lambda m}{\tr(S) - \lambda m} h(S(X), X).
\end{eqnarray*}
So the symmetric 2-tensor $ \dfrac{m}{w} \mathrm{Hess} w$ is  converging to $S=\frac{1}{2} \mathscr{L}_Y h$ on $N$.  On the other hand the quantity $\dfrac{m}{w} \mathrm{Hess} w(\xi, \xi)$ is blowing up as $m\rightarrow \infty$.

Also note that the corresponding measures
\[ w^m dvol_{g_m} \]
do not converge as $dvol_{g_m} \rightarrow dvol_g$ but $w^m = e^{L mr}$ and
\[
\abs{Lm}= \dfrac{-\lambda m}{\left(\tr(S) - \lambda m\right)^{1/2}} \rightarrow \infty.
\]
\end{rem}

\medskip
\section{Homogeneous $(\lambda, n+m)$-Einstein manifolds}

In the previous section we discussed how homogeneous  $(\lambda, n+m)$-Einstein metrics can be constructed via one-dimensional extensions.  A natural question is whether every non-trivial $(\lambda, n+m)$-Einstein metric is a one-dimensional extension of some space.  We will show this is true in the next section.   In this section we prepare for the proof of this result by applying  Theorems \ref{thm:wpedecomposition} and \ref{thm:ExistenceMBF} along with the results from \cite[Section 5]{HPWuniqueness}  to  homogeneous manifolds. We show that the base manifold of a homogeneous $(\lambda, n+m)$-Einstein manifold is either $\lambda$-Einstein or has a special form. See Lemma \ref{lem:basehomogeneous}. We also show that from a homogeneous $(\lambda, n+m)$-Einstein manifold $M$, one can find a homogeneous $\lambda$-Einstein manifold as the total space of a warped product over $M$, see Theorem \ref{thm:homogeneousE}.

First we consider base manifolds. Let $B$ be a base manifold with $\partial B = \emptyset$ and let $G$ be a transitive group of isometries acting on $B$.  Fix a point $x \in B$ and let $G_x$ be the isotropy group at $x$, i.e.,
\[
G_x = \{ \sigma \in G : \sigma(x) = x \}.
\]
Let $H$ be the subgroup which fixes  $u \in W_{\lambda, b+(k+m)}(B,g_B)$
\begin{eqnarray*}
H &=& \{ \sigma \in G : u \circ \sigma^{-1} = u \} \\
&=& G \cap \mathrm{Iso}(B, g_B)_u.
\end{eqnarray*}
Note that since $H$ is the kernel of the group homomorphism from $G$ to $\Real$, see \cite[Definition 5.3]{HPWuniqueness}, $H$ is a normal subgroup. Also note that from \cite[Proposition 5.4]{HPWuniqueness} $H$ contains $G_x$ and so we have
\[
H/G_x \subset G/G_x .
\]

\begin{lem} \label{lem:basehomogeneous}
Let $(B, g_B)$ be a base manifold which is  homogeneous and has $\partial B = \emptyset$. Then either
\begin{enumerate}
\item $(B, g_B) $ is $\lambda$-Einstein, or
\item  $\mu(u) =0$, $B$ is noncompact,  the action of $H$  on $B$ is cohomogeneity one with
\[
r: B \mapsto  B /H = \mathbb{R}
\]
where $H$ acts transitively on the level sets of $u$, $r$ is a smooth distance function and $u$ is of the form
\[
 u = Ae^{L r}
\]
for some constants $A$ and $L$.
\end{enumerate}
\end{lem}

\begin{rem}
In other words we either have  in case (1) that $ H/G_{x} = B $  and $u$ is constant, or in case (2) we have
\begin{eqnarray*}
G/G_x &=& \mathbb{R} \times H/G_x \\
g_B &=& dr^2 + g_r
\end{eqnarray*}
where $g_r$ is a family of $H$-homogeneous metrics on $H/G_x$.
\end{rem}

\begin{proof}
For $\sigma \in H$, $u \circ \sigma^{-1} =  u $. So if $H$ acts transitively on $M$, then $u$ is constant and $B$ is $\lambda$-Einstein.

Otherwise, suppose that $u$ is not constant.  Then $\mathrm{Iso}(B, g_B)_u$ must be a codimension one subgroup of $\mathrm{Iso}(B, g_B)$ and thus $H \subset G$ also has codimension one. So $H$ acts on $B$ by cohomogeneity one. In this case, from \cite[Proposition 5.4]{HPWuniqueness} we know that $B$ is noncompact and $\mu(u)=0$. Let $r$ be the quotient map
\[
r: B \mapsto B/H.
\]
Since $u$ is preserved by $H$, $u$ can be written as a function of $r$, $u=u(r)$. The fact that $D\sigma_p(\nabla u) = C  \nabla u |_{\sigma(p)}$ for any $\sigma \in G$ shows that if $u$ has a critical point, then $u$ is constant.  This shows that $B/ H$ must be all of $\mathbb{R}$.

Let $\gamma$ be a unit speed integral curve of $\nabla r$ with $\gamma(0) = x$. Define $h_s$ to be a one-parameter subgroup of isometries taking $\gamma(0)$ to $\gamma(s)$.  The differential of $h_s$ gives a Killing vector field $X^* = \nabla r + Y$ where $Y$ is tangent to the level surfaces of $u$. From \cite[Proposition 5.2]{HPWuniqueness} we have
\begin{eqnarray*}
L u = D_{X^*}  u = D_{\nabla r} u
\end{eqnarray*}
for some constant $L$.  Integrating  this implies that $u = Ae^{Lr}$.

Finally, the fact that $\lambda<0$ follows from \cite[Proposition 4.5]{HPWuniqueness}, since  $u$ is a positive function in $W_{\lambda, b+(k+m)} (B)$  and $B$  has constant scalar curvature.
\end{proof}

\begin{rem}
When $k+m>1$ we can also see that the constant $L$ is determined by the scalar curvature of $B$.  To see this we compute
\begin{eqnarray*}
\mu_B(u)  &=& (m+k-1) |\nabla u|^2 + \frac{u^2}{m+k}( \mathrm{scal}^B - (n-m-k) \lambda) \\
&=& \left( (m+k-1) L^2 + \frac{  \mathrm{scal}^B - (n-m-k)\lambda}{m+k}  \right) e^{2Lr}.
\end{eqnarray*}
Since $\mu_B(u) =0$ we obtain
\begin{equation}\label{eqn:L}
L^2 = -\frac{\mathrm{scal}^B - (n-m-k)\lambda}{(m+k)(m+k-1)}.
\end{equation}
The only case where $L$ is not determined by the above formula is when $k=0$ and $m=1$, in which case $\mu_B$ is always zero.
\end{rem}

\begin{rem}
A different proof of this lemma can be established using the results in \cite{HPWconstantscal} more heavily.  In \cite{HPWconstantscal} we also produced examples showing case (2) is possible. We generalize that  construction in the next section.
\end{rem}

Next we consider the warped product $M = B \times_u F$ where $M$ is homogeneous and $\dim W_{\lambda, n+m}(M) = k+1 > 1$. Then Theorem \ref{thm:isometryMBF} gives us the following proposition.

\begin{prop}\label{prop:homogeneousB}
Let $M$ be a simply connected metric with $\dim W_{\lambda, n+m}(M) = k+1>1$. Then $M$ is homogeneous if and only if its base manifold $B$ has no boundary and is homogeneous.
\end{prop}
\begin{proof}
From the proof of Theorem \ref{thm:isometryMBF} in \cite{HPWuniqueness} we know that isometries of $M$ preserve the singular set where all functions in $W(M, g)$ vanish, showing that if $M$ is homogeneous then the singular set is empty. It follows that $\partial B = \emptyset$.  Now Theorem \ref{thm:isometryMBF} shows that the isometry group of $M$ acts transitively on $M$ if and only if the isometry group of $B$ acts transitively on $B$.
\end{proof}

The quadratic form $\mu$ on $W(M, g)$ is either positive definite, semi-positive definite with nullity one, or non-degenerate with index one. The manifold $(M, g)$ is said to be \emph{elliptic}, \emph{parabolic} or \emph{hyperbolic} respectively, in each of the three cases, see \cite[Section 2]{HPWuniqueness}. As a simple consequence of the previous lemma and proposition, we note the following

\begin{cor}
Let $M$ be a simply connected homogeneous metric with $W_{\lambda, n+m}(M) \neq \{ 0\}$. If either $\lambda \geq 0$, or  $M$ is elliptic or hyperbolic, then $M$ is isometric to the Riemannian product $B \times F$.
\end{cor}

As another  application, we prove that if the base of a warped product Einstein metric is homogeneous, then there is  a warped product metric with the same base which is homogeneous and Einstein.

\begin{thm}\label{thm:homogeneousE}
Let $m>0$ be an integer and let $M$ be a simply connected homogeneous manifold with $W_{\lambda, n+m}(M,g) \neq \{0\}$.  If  there exists  a positive function $w \in W_{\lambda, n+m}(M,g)$ and $F^m$  is the simply connected space form with  Ricci curvature $\mu(w)$, then the warped product metric
\[ E = M \times_w F \]
is both $\lambda$-Einstein and homogeneous.
\end{thm}

\begin{proof}
The proof breaks into various cases.

We first assume that $M$ is a base manifold.   By Lemma \ref{lem:basehomogeneous} there are two cases.   First,  if $M$ is $\lambda$-Einstein then clearly taking $w=c$ a constant and $F$ a space form with Ricci curvature $\dfrac{\lambda}{c^2}$ will make $E$ a homogeneous $\lambda$-Einstein metric. On the other hand, if $M$ is a base manifold and is not $\lambda$-Einstein, then we know that   $w=Ae^{Lr}$ and $\mu(w) = 0$. In particular,  $F  = \mathbb{R}^m$.    We also have that, if  $\sigma_1 \in \mathrm{Iso}(M,g_M)$ then  $w \circ \sigma_1 = C w$  for some constant $C = C_{\sigma_1}$. By Theorem \ref{thm:isometryMBF}, or \cite[Lemma 5.6]{HPWuniqueness},  the product map $\sigma_1 \times \sigma_2 $ is an isometry of $(E,g_{E})$ where $\sigma_2$ is a $C$-homothety of $\mathbb{R}^m$.  This  gives us  a transitive group of isometries acting on $E$.

Next we assume that $M$ is not a base manifold.  If  $M$ is a space of constant curvature then the theorem is true by the special form of the warping functions, see \cite[Example 2.1]{HPWuniqueness}. Note that the sphere does not have a positive function in  $W_{\lambda, n+m}(M,g)$. Otherwise, from Lemma \ref{lem:basehomogeneous} again we also have two different cases. In the first case, we have $M=B\times \widetilde{F}$ where $B$ is $\lambda$-Einstein and $\widetilde{F}$ is a space form. We know that if $\lambda>0$, there are no positive functions in $W_{\lambda, n+m}(M)$.  When $\lambda \leq 0$, we have that  $w = Av$ where $v$ a positive function in $W_{ \lambda, k+m}(\widetilde{F})$. Take another space form $F$ such that $\widetilde{F}\times_{v} F$ is a homogeneous $\lambda$-Einstein manifold.  Then $E$ is  a product of $\lambda$-Einstein manifolds $B$ and $\widetilde{F}\times_{v} F$ which are both homogeneous. Finally, in the second case we have
\[ g_M =  g_B + e^{Lr} g_{\mathbb{R}^k}\]
where $B$ is a base manifold.   Since $w>0$ on $M$ this tells us that $w=A e^{Lr}$ for some constant $A$. Then we can write
\begin{eqnarray*}
g_{E} &=& g_B + e^{Lr}( g_{\mathbb{R}^k} + A g_{\mathbb{R}^m} ) \\
&=& g_B + e^{Lr} g_{\mathbb{R}^{k+m}}.
\end{eqnarray*}
Since $B$ is a base manifold, we can now  apply the base manifold case that we already discussed to this metric to show that $(E, g_{E})$ is homogeneous. This finishes the proof.
\end{proof}

\begin{rem}\label{rem:homogeneous}
Note that in the last case of the previous Theorem \ref{thm:homogeneousE}, when
\begin{eqnarray*}
g_M = g_B + e^{Lr} g_{\mathbb{R}^k}  & \text{ and } & w(r) =  A e^{Lr}
\end{eqnarray*}
that we still have the property that $w \circ \sigma^{-1} = C_{\sigma} w$ for any $\sigma \in \mathrm{Iso}(M,g)$, even though  $M$ is not a base manifold. Writing $M = G/G_x$ this allows us to conclude, as we did at the beginning of this section in the base manifold case, that
\[
G/G_{x} = \mathbb{R} \times H/G_{x} 
\]
with
\[
g_M = dr^2 + g_r \quad \text{and} \quad w = Ae^{L r}
\]
where $g_r$ is a family of $H$-homogeneous metrics on $H/G_{x}$.
\end{rem}

\medskip
\section{Non-trivial Homogeneous $(\lambda, n+m)$-Einstein manifolds are one-dimensional extensions}

In this section we characterize homogeneous $(\lambda, n+m)$-Einstein manifolds using one-dimensional extensions, see Theorem \ref{thm:structureWPEextension}. From this structure theorem we prove Theorem \ref{thm:WPE1extensionintro}, see Theorem \ref{thm:WPEextDnormal}, and Corollary \ref{cor:SARiccisiltonDnormal}. 

\begin{thm}\label{thm:structureWPEextension}
Let $(M^n,g)$ be a homogeneous $(\lambda, n+m)$-Einstein space which is not Einstein.  Then $(M,g)$ is the one-dimensional extension of a homogeneous space $(N^{n-1},h)$ with a derivation $D$ and $\alpha^2 = \dfrac{1}{\tr(S) - \lambda m}$ satisfying the following conditions.
\begin{enumerate}
\item $\Ric^N = \lambda I + S + \dfrac{1}{\tr(S)-\lambda m} [S,A] $
\item $\mathrm{div}(S) = 0$
\item $\tr(S^2) = - \lambda \tr(S)$,
\end{enumerate}
where $\lambda<0$. Moreover the warping function $w \in C^\infty(M)$ is given by $w(r) = e^{L r}$ where $L = \lambda \alpha$ and $r$ is the signed distance function to $N$.
\end{thm}

\begin{proof}
Let $G$ be a connected Lie group that acts transitively on $M$ by isometries. Fix a point $x \in M$ and let $K = G_x$ be the isotropy subgroup at $x$. The Lie algebras of $G$ and $K$ are denoted by $\frakg$ and $\frakk$ respectively. Let $\frakq$ be the orthogonal complement of $\frakk$ in $\frakg$. The tangent space $T_x M$ is identified with $\frakq$. In the following we separate our argument into two different cases.

\smallskip

\textsc{Case I}. We assume that $\dim W(M, g) =1$, i.e., $(M,g)$ is a base manifold. From Lemma \ref{lem:basehomogeneous} we may assume that $w(r) = e^{Lr}$ for some constant $L$. From the proof of Lemma \ref{lem:basehomogeneous} there is a codimension one normal subgroup $H\subset G$ containing $K$ that acts transitively on the hypersurface $N^{n-1} = G.x$. Let $\frakh$ be the Lie algebra of $H$ and $\xi \in \frakg$ corresponds to $\nabla r$ at the point $x$. Let $\frakp\subset \frakq$ be the orthogonal complement of $\xi$, then we have
\[
\frakg = \Real\xi \oplus \frakp \oplus \frakk, \qquad \frakh = \frakp\oplus \frakk.
\]
Since $H \subset G$ is a normal subgroup, we have $[\xi, X]\in \frakh$ for any $X\in \frakh$. It follows that $[\xi, \cdot]$ defines a derivation on $\frakh$ and thus $M$ is a one-dimensional extension of $N$. Following the construction in Theorem \ref{thm:WPE1extension} we define
\begin{equation}\label{eqnDXadxi}
D(X) =  \frac{1}{\alpha} \ad_\xi (X) \quad \text{for all } X \in \frakh
\end{equation}
where $\alpha =  L/ \lambda$. In the following we show that $(N, h)$ with the derivation $D$ satisfies the properties stated in the theorem.

From Proposition \ref{prop:1extensionST}, the shape operator of $N \subset M$ at $x$ is given by $T = - \alpha S$ where $S$ is the symmetrization of $D$ in equation (\ref{eqn:symmD}).

\begin{claim}
We have $\tr(S^2) = - \lambda \tr(S)$.
\end{claim}
In fact since $w = e^{Lr}$ for any $X\in T_x N$ we have
\[
\Hess w(X, X) = L w h(T(X), X).
\]
From the equation $\Ric^M - \frac{m}{w}\Hess w = \lambda g$ we have
\[
T(X) = \frac{1}{mL}\left(\Ric^M(X) - \lambda X\right).
\]
It follows that
\begin{eqnarray*}
\left(\nabla_\xi T\right)(X) & = & \frac{1}{m L}\left(\nabla_\xi \Ric^M\right)(X) = \frac{1}{m L}\left(\nabla_{\frac{\nabla w}{L w}}\Ric^M \right)(X) \\
& = & \frac{1}{m L^2 w}\left(\nabla_{\nabla w}\Ric^M\right)(X).
\end{eqnarray*}
On the other hand, from \cite[Proposition 3.7]{HPWconstantscal} we have
\[
\left(\nabla_{\nabla w}\Ric^M\right)(X, Y) = \frac{w}{m}g\left(\left(\Ric^M - \rho I\right)(X), \left(\Ric^M - \lambda I\right)(Y)\right) + \frac{m}{w}Q(\nabla w, X, Y, \nabla w),
\]
where $Q$ is the $(0,4)$-tensor defined at the beginning in \cite[Section 3]{HPWconstantscal} and
\[
\rho = \frac{(n-1)\lambda - \mathrm{scal}}{m-1}.
\]
In terms of $T$, the above equation can be written as
\[
m L \left(\nabla_{\nabla w} T\right)(X, Y) = \frac{w}{m}g\left(mL T(X) + (\lambda -\rho) X, mL T(Y)\right) + \frac{m}{w}Q(\nabla w, X, Y, \nabla w).
\]
Let $\set{X_i}_{i=1}^{n-1}$ be an orthonormal basis of $T_x N$ and then we have
\[
mL\sum_{i=1}^{n-1}\left(\nabla_{\nabla w} T\right)(X_i, X_i) = w m L^2 \tr(T^2) + w L (\lambda-\rho)\tr(T).
\]
The trace of the term that involves $Q$ vanishes by using equation (3.1) and Proposition 3.3 in \cite{HPWconstantscal}. From Proposition \ref{prop:1extensionST}, since $\nabla_\xi T = \alpha^2 [S,A]$, the left hand side of the equation above is zero and we have
\[
m L \tr(T^2) + (\lambda -\rho)\tr(T) = 0.
\]
Note that $k = 0$ in equation (\ref{eqn:L}) and we have
\[
L^2 = - \frac{\mathrm{scal} - (n-m)\lambda}{m(m-1)} = \frac{\rho - \lambda}{m}.
\]
It follows that $\tr (T^2) = L \tr(T)$. Then the claim follows by $T = - \alpha S$ and $L =\lambda \alpha$.

From the formulas of Ricci curvature in Lemma \ref{lem:RicciGKHK} we have
\begin{eqnarray*}
\left(\Ric- \frac{m}{w}\Hess w\right)(\xi, \xi) & = &  - \alpha^2 \tr(S^2) - mL^2 \\
\left(\Ric- \frac{m}{w}\Hess w\right)(\xi, X) & = & - \alpha \mathrm{div}(S) \\
\left(\Ric- \frac{m}{w}\Hess w\right)(X, X) & = & \Ric^N(X, X) - \left(\alpha^2 \tr(S) - mL \alpha\right)h(S(X), X) \\
& & - \alpha^2 h\left([S,A](X),X \right).
\end{eqnarray*}
The second equation above implies that
\[
\mathrm{div}(S) = 0.
\]
The first equation shows that
\[
- \alpha^2 \tr(S^2) - mL^2 = \lambda
\]
i.e.,
\begin{eqnarray*}
\lambda = \lambda \alpha^2 \tr(S) - mL^2 = \lambda \alpha^2 \tr(S) - m \lambda^2 \alpha^2
\end{eqnarray*}
and so we have
\[
\alpha^2(\tr(S) - m\lambda) = 1.
\]
Plugging it into the third equation above shows that
\begin{eqnarray*}
\lambda I & = & \Ric^N - \left( 1+ m\lambda\alpha^2 - m \lambda \alpha^2\right)S - \alpha^2 [S,A] \\
& = & \Ric^N - S - \frac{1}{\tr(S) - m\lambda}[S,A].
\end{eqnarray*}
It follows that
\[
\Ric^N = \lambda I + S + \frac{1}{\tr(S) - m\lambda}[S, A]
\]
which finishes the proof in this case.

\smallskip

\textsc{Case II.} We assume that $\dim W(M, g)=k+1$ with $k\geq 1$. From Lemma \ref{lem:basehomogeneous} and Proposition \ref{prop:homogeneousB} we know that $M = B^b \times_u \Real^k$ and $B$ is a homogeneous base manifold and $u \in W_{\lambda, b+(k+m)}(B,g_B)$ is positive everywhere with $\mu_B(u) = 0$. From the proof of Theorem \ref{thm:homogeneousE} we may assume that $w = u = e^{Lr}$ for some constant $L$. Let $B = G_1/ K_1$ and $K_1 \subset H_1 =G_1 \cap \mathrm{Iso}(B, g_B)_u$. From Theorem \ref{thm:isometryMBF} the group $G = G_1 \ltimes \Real^k$ acts transitively on $M$ via isometries where $\Real^k$ is the $C$-translation in the Euclidean group, i.e., $A = \mathrm{Id}$ in Theorem \ref{thm:isometryMBF}. The isotropy subgroup is given by $K = K_1 \times \set{0}$. Applying the argument in the previous case to $(B, g_B)$ yields that $B$ is a one-dimensional extension of $N_1^{b-1}= H_1/ K_1$ with a derivation $D \in \Der(\frakh_1)$ and a constant $\alpha$ such that the following equations hold.
\begin{eqnarray}
\tr_{\frakh_1}(S^2) & = & -\lambda \tr_{\frakh_1}S \label{eqn:trS2S} \\
\mathrm{div}_{N_1} S & = & 0 \label{eqn:divS} \\
\Ric^{N_1} & = & \lambda I + S + \frac{1}{\tr_{\frakh_1}S - (k+m)\lambda}[S, A]. \label{eqn:RicN1SA}
\end{eqnarray}

Let $H = H_1 \times \Real^k$ such that $N = H/K = N_1\times \Real^k$ with the product metric is the zero-level set of $w$ in $M$. It follows that $\frakh = \frakh_1 \oplus \Real^k$. We extend $D$ to $\frakh$ by letting $D(U) = - \lambda U$ for any $U \in \Real^k$ and the Lie bracket is extended by $[\xi, U] = - L U$ and $[X, U] = 0$ for any $X \in \frakh_1$. It is easy to verify that $D$ is a derivation of $\frakh$. Moreover we have $\tr S = \tr_{\frakh_1} S - k \lambda$ on $\frakh$ and so equation (\ref{eqn:RicN1SA}) can be written as
\[
\Ric^N|_{N_1} = \lambda I + S + \frac{1}{\tr(S) - m \lambda}[S,A].
\]
Since $S = - \lambda I$ and $A = 0$ on $\Real^k$, the right hand side of above equation is zero which is equal to $\Ric^N$ on the $\Real^k$ factor. This shows property (1) in the theorem. Property (3) follows from equation (\ref{eqn:trS2S}) and the extension of $D$ to $\frakh$. Finally property (2) follows from equation (\ref{eqn:divS}), the product structure of $N$ and the fact that $S = - \lambda I$ on the $\Real^k$ factor.
\end{proof}

We can now prove the converse statement of Corollary \ref{cor:SolitonQuasiEinstein}.

\begin{thm}\label{thm:WPEextDnormal}
A non-flat, non-trivial semi-algebraic Ricci soliton on a homogeneous space admits a $(\lambda,n+m)$-Einstein one-dimensional extension with 
\[
\ad_\xi (X) = \alpha D(X), \qquad \text{for all } X \in \frakh
\]
if and only if $D$ is normal.
\end{thm}
\begin{proof}
Let $(N^{n-1}, h)$ be a semi-algebraic Ricci soliton
\[
\Ric = \lambda I + S
\]
for some constant $\lambda < 0$. As in the proof of Theorem \ref{thm:EinsteinExtDnormal}, we have 
\[
\mathrm{div} S = 0,\quad \tr(S^2) = - \lambda \tr(S) \quad \text{and}\quad \abs{\Ric}^2 = \lambda \scal.
\]
If $D$ is normal, then $[S,A] = 0$ and from Theorem \ref{thm:WPE1extension} there is a one-dimensional extension of $(N, h)$ which is $(\lambda, n+m)$-Einstein. 

On the other hand, if $(N, h)$ admits a one-dimensional extension which is $(\lambda, n+m)$-Einstein, then from Theorem \ref{thm:structureWPEextension} on $N$ we have
\[
\Ric = \lambda I + S + \frac{1}{\tr(S) - \lambda m}[S, A].
\]
Comparing with the semi-algebraic Ricci soliton equation yields 
\[
[S, A] = 0,
\]
i.e., $D$ is normal. 

To finish the proof we show that if the extension is $(c, n+m)$-Einstein with $c\ne \lambda$, then $N$ is flat. From Theorem \ref{thm:structureWPEextension} again we have $\tr(S^2) = - c \tr (S)$ and thus we have $(c-\lambda)\tr(S) = 0$. It follows that $\tr(S) = 0$ and then $\tr(S^2) = 0$, i.e., $S = 0$ which shows that $N$ is a trivial semi-algebraic Ricci soliton, a contradiction. 
\end{proof}

As a consequence of Theorem \ref{thm:structureWPEextension} we also have the following formula for the covariant derivative of the Ricci tensor in the direction of $\nabla w$. 

\begin{prop} \label{prop:RadialRic} If $(M^n,g)$ is a homogeneous $(\lambda, n+m)$-Einstein metric, then 
\[ 
\nabla_{\nabla w} \mathrm{Ric} = (m w L \alpha^2)[S,A]. 
\]
\end{prop}

\begin{proof}
By Theorem \ref{thm:structureWPEextension}  $M$ is a one-dimensional extension of a homogeneous space $N$ and
\[ w= e^{Lr}, \]
where $r$ is the distance to $N$. This tells us that
\[  \frac{1}{w} \mathrm{Hess} w = L^2 dr \otimes dr + L g(T(\cdot), \cdot). \]
Taking the covariant derivative of both sides of the equation above gives us,
\begin{eqnarray*}
&\nabla_{\nabla w}  \left( \dfrac{1}{w}\mathrm{Hess} w  \right) = L e^{Lr} g\left( \left(\nabla_{\nabla r} T\right) (\cdot), \cdot \right), &
\end{eqnarray*}
where we have used that $\nabla_{\nabla r} \left( dr \otimes dr \right) = 0$, which follows from a simple calculation since $r$ is a distance function. 

Then by differentiating the $(\lambda, n+m)$-Einstein equation we have that 
\begin{eqnarray*}
\nabla_{\nabla w} \left(\mathrm{Ric}  \right)  &=& m \nabla_{\nabla w}  \left( \frac{1}{w}\mathrm{Hess} {w} \right) \\
&=& mLw \left(\nabla_{\nabla r} T\right)\\
&=& mLw \alpha^2 [S,A]. 
\end{eqnarray*}
In the last line we have used Proposition \ref{prop:1extensionST}.
\end{proof}

This result, when combined with our earlier results, gives us the following results mentioned in the introduction. 

\begin{proof} [Proof of Theorem \ref{thm:WPEradial}]
From  Proposition \ref{prop:RadialRic},  $\nabla_{\nabla w} \mathrm{Ric}=0$ if and only if  $[A,S]=0$.  When $[A,S]=0$, part  (1) of Theorem \ref{thm:structureWPEextension} shows that  $M$ is the one-dimensional extension of a normal semi-algebraic Ricci soliton. 
\end{proof}

\begin{proof}[Proof of Corollary \ref{cor:WPEradial}]
If $w$ is constant, then the statement is trivial.  If $w$ is non-constant, then we have that 
\[ g_E = g_M +  w^2 g_{\mathbb{R}^m}\]
where $w>0$.  Moreover, $g_M$ is a one-dimensional extension of $(N,g_N)$,  a normal semi-algebraic Ricci soliton.  

By Theorem \ref{thm:EinsteinConstruction}, $N$ also  admits a different one-dimensional extension which is $\lambda$-Einstein.  The underlying manifold of this Einstein metric is also $M= N \times \mathbb{R}$.  Denote this metric by $\widetilde{g}_M$. Then the metric 
\[ \widetilde{g}_E =  \widetilde{g}_M + g_{H^m} \]
is  a metric on $E$ where $g_{H^m}$ denotes the $m$-dimensional hyperbolic space with Ricci curvature $\lambda$.  
\end{proof}

We end this section with the proof of  Corollary \ref{cor:SARiccisiltonDnormal} in Introduction. 

\begin{proof}[Proof of Corollary \ref{cor:SARiccisiltonDnormal}]
The case when $m=0$ follows from Theorem \ref{thm:EinsteinExtDnormal}. When $m\geq 1$, from Theorem \ref{thm:WPEextDnormal} we know that $N^{n-1}$ admits a homogeneous one-dimensional extension $M^n$ which is $(\lambda, n+m)$-Einstein. From Theorem \ref{thm:homogeneousE} the warped product of $M$ with a space form fiber $F^{m}$ is both homogeneous and $\lambda$-Einstein. 
\end{proof}

\medskip
\section{Left invariant metrics on Lie groups and algebraic Ricci solitons}

In this section we specialize to Lie groups with left invariant metrics. In the first subsection we discuss general results about how  $W(G,g)$  interacts with the Lie group structure of $G$. In the second subsection we consider simply connected solvable Lie groups with left invariant metrics and give a classification of such groups that have $W(G,g) \neq \{ 0 \}$.

\subsection{Left invariant metrics on  Lie groups with  $\mathbf{W \neq 0}$.} Let $(G, g)$ be an $n$-dimensional Lie group with left invariant metric such that $W_{\lambda, n+m}(G, g) \ne \set{0}$. Combining Lemma \ref{lem:basehomogeneous} and Proposition \ref{prop:homogeneousB}, we have two cases.  Either
\begin{enumerate}
\item $(G,g)$ is isometric to a Riemannian product $ B \times F^k$   where $B$ is a Lie group with left invariant $\lambda$-Einstein metric and $F$ is a simply connected space form, or
\item
\[ g = g_B + e^{Lr} g_{\Real^k} \]
 where $L$ is a constant,  $B$ is a base manifold,  $r : B \To \Real$ is a smooth distance function and
\[ w = e^{Lr} \in W_{\lambda, n+m}(G,g)\]
where $\lambda<0$.
\end{enumerate}
Note that, in either case, $k$ could be zero. In case (1) we call the metric $(G,g)$ \emph{rigid} and, in  case (2) we call the metric \emph{non-rigid}.

In the non-rigid case  we can always re-parametrize the distance function $r$ so that $r(e)=1$, where $e$ is the identity element of $G$.  We call such a re-parametrized distance function \emph{normalized}.   We will usually assume the distance function is normalized below.

The study of $W(G,g)$ in the rigid case reduces to studying the  solutions on space forms discussed in \cite[Section 2]{HPWuniqueness} as $W$ will consist of  the pullbacks of functions $v \in W_{\lambda, k+m}(F, g_F)$.  In the second case we have the following more interesting interaction between properties of the function $u$ and the Lie group structure.

\begin{prop} \label{prop:H} 
Let $(G,g)$ be a non-rigid  Lie group with left invariant metric  and let   $w \in W_{\lambda, n+m}(G,g)$ where
\[ w = e^{Lr}, \]
$L$ is a constant, and $r$ is a normalized distance function. Then  $r$ is the signed distance to a codimension one normal subgroup $H$ and the vector field $\xi = \nabla r $ is a left invariant vector field.  In particular,
\begin{equation}
\label{eqn:xiH}  [\xi,  \mathfrak{h}] \subset \mathfrak{h}
\end{equation}
where $\mathfrak{h}$ is the Lie algebra of $H$.
\end{prop}

\begin{proof}
Let $H$ be the level hypersurface $w=1$. Since $w(e)=1$, the elements of $H$ are the elements whose left translation preserves $w$.  In particular, $H$ is  a codimension one normal subgroup in $G$. Thus we obtain a Riemannian submersion $G\rightarrow G/H$
which is also a Lie algebra homomorphism. This shows that left invariant
vector fields on $G/H$ lift to left invariant vector fields on $G$
that are perpendicular to $H$. As $G/H=\mathbb{R}$ it follows that
$\nabla r$ is a left invariant vector field on $G.$

For the last part note that
\[ g([\xi, X], \xi) = -2g (\nabla_{\nabla r} \nabla r, X) = 0 \]
for any $X \perp \xi$.
\end{proof}

In the spirit of Theorem \ref{thm:homogeneousE}, we  now address the question of whether it is always possible to build a left invariant Einstein metric from a Lie group with left invariant metric $(G,g)$ with $W(G,g) \neq 0$.

In the rigid case we can see quickly  that this is always true.  Given  $v \in W_{\lambda, k+m}(F, g_F)$, let  $\widetilde{F}$ be the $m$-dimensional simply connected space form with Ricci curvature $\mu_{F}(v)$.  Define
\begin{eqnarray*}
E &=& G \times_{v} \widetilde{F}.
\end{eqnarray*}
Then we have
\begin{eqnarray*}
E &=& B \times (F \times_{v} \widetilde{F}) \\
&=& B \times \hat{F}
\end{eqnarray*}
where $\hat{F}$ is a simply connected space form.  Clearly, $E$ is naturally a Lie group with the product structure coming from $B$ and $\hat{F}$, and  the metric, being a product of left invariant metrics on the factors, is  also left invariant.

In the non-rigid case we also have the following

\begin{thm} \label{thm:LieStruc}
Suppose that $(G,g)$ is a non-rigid  Lie group with left invariant metric   such that
\[
w = e^{Lr} \in W_{\lambda, n+m}(G,g),
\]
where $L$ is a constant, and $r$ is a normalized distance function.  Let $E = G\times_{w}\Real^m$ with metric $g_E = g + w^2 g_0$ where $(\Real^m, g_0)$ is  Euclidean space. Then $E$ is a Lie group and its Lie algebra $\mathfrak{e}$ is the abelian extension of the Lie algebra $\mathfrak{g}$ of $G$ by $\mathfrak{a} = \Real^m$ with
\[
[\xi, U] = - L U, \quad [X_i, U] = 0 \text{ for any } U \in \mathfrak{a} \text{ and }i = 1, \ldots, n-1,
\]
where $\set{\xi}\cup \set{X_i}_{i=1}^{n-1}$ is an orthonormal basis of the metric Lie algebra  $\mathfrak{g}$.
\end{thm}
\begin{proof}
From Remark \ref{rem:homogeneous}, in the non-rigid case we know that, for any $x \in G$ there is a constant,  $C_x$,  such that $u \circ L_{x^{-1}} = C_x u$ where $L_{x^{-1}}$ the left multiplication of $x^{-1}$ on $G$.   $C_x$ induces  an automorphism $\tau(x, \cdot)$ of $(\Real^m, +)$ in the following way
\[
\tau(x, \cdot) : \Real^m \rightarrow \Real^m, \quad a \mapsto C_x a.
\]
For a fixed $x \in G$, the differential $\bar{\tau}(x)$ of $\tau(x)$ is a Lie algebra isomorphism of the abelian Lie algebra $\mathfrak{a}$ which is the tangent space of the Lie group $(\Real^m, +)$ at the origin. In particular $\mathfrak{a}$ is isomorphic to $(\Real^m, +)$ and we have
\[
\bar{\tau}(x) : \mathfrak{a} \rightarrow \mathfrak{a}, \quad U \mapsto C_x U.
\]
From the formula $w(x) = e^{L r(x)}$, we also  have
\[
C_x = \frac{w(e)}{w(x)} = \frac{1}{e^{L r(x)}} = e^{- L r(x)}.
\]
On the total space $E$, a Lie group structure is given by the semidirect product $G\times_{w}\Real^m $ which is the Lie group with $G \times \Real^m$ as its underlying manifold and with multiplication and inversion given by
\begin{eqnarray*}
(x,a) \cdot (y, b) & = & (x\cdot y, C_{y^{-1}}a + b) \\
(x,a)^{-1} & = & (x^{-1}, - C_x a).
\end{eqnarray*}
For the semi-direct product of two general Lie groups, see \cite[Section I.15]{Knapp} for example.

The map $\bar{\tau}$ is a smooth homomorphism of $G$ into $\mathrm{Aut}(\mathfrak{a})$, the automorphisms of $\mathfrak{a}$. The differential $D \bar{\tau}$ is a homomorphism of the Lie algebra $\mathfrak{g}$ of $G$ into $\mathrm{Der}(\mathfrak{a})$, the derivations of $\mathfrak{a}$. The Lie algebra $\mathfrak{e}$ of $E$ is given by the semi-direct product $\mathfrak{g}\oplus_{D\bar{\tau}}\mathfrak{a}$, i.e., the Lie brackets of $\mathfrak{g}$ and $\mathfrak{a}$ are preserved in $\mathfrak{e}$ and, for any $X \in \mathfrak{g}$, $U \in \mathfrak{a}$ we have
\[
[X, U] = \left(D{\bar{\tau}}(X)\right)(U).
\]
In the following we compute the map $D\bar{\tau}$.

Let $\set{X_i}_{i=0}^{n-1}$ be an orthonormal basis of $\mathfrak{g}$ with $X_0 = \xi = \nabla r|_{e}$. For $t \in \Real$ let $x(t) = \exp(tX_i)$. If $i \geq 1$, then $x(t) \in H$ and it follows that $r(x(t)) = 0$. Thus $C_{x(t)} = 1$ and $\bar{\tau}(x(t))$ is the identity map for $t \in \Real$. So its differential is zero, i.e., $[X_i, U] = 0$. Now we are left with $D\bar{\tau}(X_0)$. In this case we have $r(x(t)) = t$ and then
\begin{eqnarray*}
\left(D\bar{\tau}(X_0)\right)(U) = \frac{\df}{\df t}\left(e^{-L t}U\right) = -L U,
\end{eqnarray*}
which shows that $[\xi, U] = -L U$.
\end{proof}

\begin{rem}\label{rem:solvablility}
From the Lie algebra structure of $\frake$ in Theorem \ref{thm:LieStruc}, we have
\[
\frake^1 = [\frake, \frake] = [\frakg, \frakg] \oplus \Real^m = \frakg^1 \oplus \Real^m,
\]
and
\[
\frake^2 = [\frake^1, \frake^1] = [\frakg^1, \frakg^1] = \frakg^2
\] 
as $\frakg^1\subset \frakh$ by Proposition \ref{prop:H}. On the other hand from Proposition \ref{prop:H} again we have the following relation in the commutator Lie algebras 
\[
\frakh^{i+1} \subset \frakg^{i+1} \subset \frakh^{i}
\]
for $i \geq 0$. It follows that $\frake$ is a solvable Lie algebra if and only if $\frakh$ is solvable.  
\end{rem}

\smallskip

\subsection{Non-rigid  solvable Lie groups with $\mathbf{W \neq 0}$. }
From Theorems \ref{thm:WPE1extension} and \ref{thm:WPEextDnormal}, a one-dimensional extension of an algebraic Ricci soliton $(H, h)$ admits a non-rigid $(\lambda, n+m)$-Einstein metric. We show a converse of this construction on solvmanifolds, i.e., any non-rigid $(\lambda, n+m)$-Einstein solvmanifold can be obtained in this way.

A solvmanifold with $W_{\lambda, n+m}(G,g) \neq \{0\}$ which is rigid  is a  product of a $\lambda$-Einstein  solvmanifold and a space form.   Moreover,  $W_{\lambda, n+m}(G,g)$ consists of functions which are pullbacks of solutions  on the space form factor.  Thus, the study of these spaces  reduces to studying  left invariant Einstein metrics on simply connected solvable Lie groups.  There is a rich structure to these spaces, see \cite{Heber},  \cite{LauretStandard} and the references therein.

In the non-rigid case,  the group $G$ shall be identified with its metric Lie algebra $(\mathfrak{g}, \langle\cdot, \cdot\rangle)$ where $\mathfrak{g}$ is the Lie algebra of $G$ and $\langle\cdot, \cdot\rangle$ denotes the inner product on $\mathfrak{g}$ which determines the metric. We consider the orthogonal decomposition
\begin{equation*}
\mathfrak{g} = \mathfrak{a}\oplus \mathfrak{n},
\end{equation*}
where $\mathfrak{n}$ is the nilradical of $\mathfrak{g}$, i.e., the maximal nilpotent ideal.   Assuming that $(G,g)$ is non-rigid,  by Proposition \ref{prop:H},  the zero set of a normalized distance function $H$  is a codimension one normal subgroup.  Let $h$ be the induced metric on $H$.    Then $(H,h)$  is also a solvmanifold since, by equation (\ref{eqn:xiH}), $\xi \in \mathfrak{a}$.  The Lie algebra  of $H$, $\mathfrak{h}$, then  has the following decomposition
\[
\mathfrak{h} = \mathfrak{a}' \oplus \mathfrak{n}
\]
where $\mathfrak{a}'$ is the orthogonal complement of $\Real \xi \subset \mathfrak{a}$.

\begin{thm}\label{thm:Hsolvsoliton}
Suppose that $(G, g)$ is a solvmanifold  with $W_{\lambda, n+m}(G, g) \neq \{ 0 \}$ which is non-rigid. Let $H$ be  the zero set of a normalized distance function with induced metric $h$. Then  $(H, h)$ is an algebraic Ricci soliton with $\mathrm{Ric}^{H} = \lambda I + S$ where $S \in \mathrm{Der}(\mathfrak{h})$ is symmetric and $\mathfrak{h}$ is the Lie algebra of $H$.
\end{thm}

\begin{rem} Recall that, under the  hypothesis, $\lambda$ must be negative. \end{rem}

\begin{proof}
From Proposition \ref{prop:H} and Theorem \ref{thm:structureWPEextension} we know that $H$ is a codimension one normal subgroup of $G$ and its Ricci curvature satisfies the following equation
\[
\Ric^H = \lambda I + S + \frac{1}{\tr(S) - \lambda m}[S, A],
\]
where the derivative $D$ is given in the proof of Theorem \ref{thm:structureWPEextension}. To show that $(H, h)$ is an algebraic Ricci soliton, it is sufficient to show that $S$, the symmetrization of $D$, is also a derivation, and $[S, A] = 0$. Since $D$ is given by $\dfrac{1}{\alpha} \ad_\xi$ in Case I of Theorem \ref{thm:structureWPEextension}, and by $\dfrac{1}{\alpha}\ad_\xi$ on $\frakh_1$ and $- \lambda I$ on $\Real^k$ in Case II of Theorem \ref{thm:structureWPEextension}, we only have to show that $(\ad_{\xi})^*$ is a derivation and $\ad_\xi$ is a normal operator. We prove this by considering the Einstein solvmanifold $(E, g_E)$ in Theorem \ref{thm:LieStruc}. 

Recall that the derivation $\ad_\xi$ is extended to $\mathfrak{e}= \Real\xi \oplus \frakh \oplus \Real^m$ by $\ad_\xi (U) = - L U$ for any $U \in \Real^m$. The property that it is a normal operator and its adjoint is a derivation of $\frakh$ holds if and only if its extension has the same property on $\mathfrak{e}$. In Theorem \ref{thm:LieStruc}, we have $[\mathfrak{e}, \mathfrak{e}] = \frakn \oplus \Real^m$ and its orthogonal complement is $\fraka$ which is abelian. It follows that $(E, g_E)$ is of standard type. Since $\xi \in \fraka$, $\ad_\xi$ is a normal operator by Theorem B, or Theorem 4.10 in \cite{Heber}. In \cite[Lemma 4.7]{LauretSol}, it is shown that this it is equivalent to $(\ad_\xi)^*$ being a derivation. This finishes the proof.
\end{proof}

Finally, using the structure results of algebraic Ricci soliton on solvmanifolds in \cite{LauretSol},  we have the following characterization of non-rigid  solvmanifolds.

\renewcommand{\theenumi}{\roman{enumi}}

\begin{thm}\label{thm:wpesolvmanifold}
Let $(G, g)$ be a solvmanifold with metric Lie algebra $(\mathfrak{g}, \scp{\cdot, \cdot})$ and consider  orthogonal decompositions of the form  $\mathfrak{g} = \mathfrak{a} \oplus \mathfrak{n}$ and $\mathfrak{a} = \Real \xi \oplus \mathfrak{a}'$, where $\mathfrak{n}$ is the nilradical of $\mathfrak{g}$ and $r$ is a signed distance function with $\nabla r = \xi$.  Then $(G, g)$ is a non-rigid space with  $e^{L r} \in W_{\lambda, n+m}(G,g)$ from some constants $L$ and $m$  if and only if the following conditions hold:
\begin{enumerate}
\item $(\mathfrak{n}, \scp{\cdot, \cdot}|_{\mathfrak{n}\times \mathfrak{n}})$ is a nilsoliton with Ricci operator $\mathrm{Ric}_1 = \lambda I + D_1$, for some $D_1 \in \Der(\mathfrak{n})$,
\item $[\mathfrak{a}, \mathfrak{a}] = 0$,
\item $(\mathrm{ad}_A)^* \in \Der(\mathfrak{g})$(or equivalently, $[\mathrm{ad}_A, (\mathrm{ad}_A)^*] = 0$) for all $A \in \mathfrak{a}$,
\item $\scp{A,A} = - \frac{1}{\lambda} \tr S(\mathrm{ad}_A)^2$ for all $A\in \mathfrak{a}'$,
\item $\tr S(\mathrm{ad}_{\xi})^2 = -\lambda - m L^2$.
\end{enumerate}
\end{thm}

\renewcommand{\theenumi}{\arabic{enumi}}

\begin{proof}
From Theorems \ref{thm:WPE1extension}, \ref{thm:WPEextDnormal} and \ref{thm:Hsolvsoliton}, $(G, g)$ is a non-rigid space with  $e^{Lr} \in W_{\lambda, n+m}(G,g)$ if and only if $(H, h)$ is an algebraic Ricci soliton, i.e., $\mathrm{Ric}^{H} = \lambda I + D$ for some $D \in \Der(\mathfrak{h})$, and $S(\mathrm{ad}_{\xi}) = \alpha D$ for some $\alpha \in \Real$. From \cite[Theorem 4.8]{LauretSol}, the structure results of algebraic Ricci solitons on solvmanifolds, we have conditions (i), (ii), (iii) and (iv) for any $\mathfrak{h}$. In (iii) the fact that $\left(\mathrm{ad}_{\xi}\right)^* \in \mathrm{Der}(\mathfrak{g})$ follows from that $S\left(\ad_{\xi}\right)$ is a derivation. The last condition (v) follows from the facts that $\mathrm{Ric}(\xi, \xi) = -(\lambda + mL^2)$ and $\mathrm{Ric}(\xi, \xi) = - \tr S(\mathrm{ad}_{\xi})^2$. It is equivalent to the existence of $\alpha$ such that $S\left(\mathrm{ad}_{\xi}\right) = \alpha D$.
\end{proof}

\medskip

\appendix

\section{An alternative approach to Semi-Algebraic Ricci solitons}

In this appendix we give an alternative approach to semi-algebraic Ricci solitons in terms of Lie derivatives acting on vector fields.  Since this approach does not rely on the additional homogeneous structure of the manifold, it readily generalizes to a concept of a general Ricci soliton being semi-algebraic with respect to a given sub-algebra of vector fields.  With this definition we can see that any (not necesarily homogeneous) Ricci soliton is semi-algebraic with respect to the algebra of Killing vector fields, generalizing the result in the homogeneous case.   

First we consider the homogeneous case.  Let $g$ be a $G$-homogeneous metric on $M$ and  let $\frakg$ be the Lie algebra of $G$.  For each $Y\in \frakg$ the one-parameter subgroup on $G$ generated by $X$,  $\exp(tY)$,  defines a one-parameter group of diffeomorphisms $\phi_t$ on $M$.  We can then identify $Y$ with the vector field on $M$ generated by $\phi_t$.  In doing so, we identify $\frakg$ with the  sub-algebra of $\mathfrak{X}(M)$ consisting of  Killing vector fields on $(M,g)$ which generate one-parameter subgroups of $G$ (recall that $G$ is not necessarily the whole isometry group).  With these identifications we obtain the following equivalent criteria for a Ricci soliton to be semi-algebraic with respect to $G$. 

\begin{prop} \label{prop:SASintrinsic}
A G-homogeneous metric $(M,g)$ is a semi-algebraic Ricci soliton if and only if there is a smooth vector field $X$ on $M$ such that 
\begin{equation}  \mathrm{Ric}_g = \lambda g +  \frac{1}{2}\mathscr{L}_X g \label{eqn:soliton} \end{equation}
and  the Lie derivative acting on vector fields
\[ \mathscr{L}_X: \mathfrak{X}(M) \rightarrow \mathfrak{X}(M) \]
leaves $\frakg \subset \mathfrak{X}(M)$ invariant.   Moreover,  for a fixed point $x$, we can assume that  $X|_x =0$ and $\mathscr{L}_X$ preserves $\frakg_x = \{ Y \in \frakg : Y|_x=0\}$. 
\end{prop}

\begin{proof}
With the respect to the fixed point  $x\in M$, with $K$ denoting the isotropy at $x$ and $Aut(G)^K$ denoting the automorphisms of $G$ that preserve $K$, using the proofs of Propositions 2.2 and 2.3 in \cite{Jablonski}, we can write a semi-algebraic Ricci soliton in the form
\[ g_t = (1-2\lambda t) \phi^*_{s(t)} (g) \]
where $\phi_t$ is defined by  
\[ \phi_s(h \cdot x) = \Phi_s(h) \cdot  x \qquad h \in G \]
and  $\Phi_s = \exp(sD)  \in Aut(G)^{K}$ for some $D \in \Der(\frakg)$.  Letting $X = \frac{d}{ds}|_{s=0} \phi^*_s$.  Then equation (\ref{eqn:soliton}) is satisfied.  Moreover, under the identification of  vector fields with elements of $\frakg$ we can see that 
\begin{eqnarray*}
\mathscr{L}_X Y = D(Y).
\end{eqnarray*}
In particular, $\mathscr{L}_X$  is a derivation when restricted to $\frakg \subset \mathfrak{X}(M)$ and, in particular must preserve $\frakg$. 

Conversely, suppose that we have a solution to (\ref{eqn:soliton}) such that $\mathscr{L}_X$ also preserves $\frakg$.  Fix a point $x \in M$, since $G$ acts transitively on $M$, there is a Killing vector field $Y \in \frakg$ such that $Y|_x = X|_x$.  Replacing $X$ by $X-Y$ allows us to assume that $X|_x = 0$.    Moreover, since $Y \in \frakg$, $\mathscr{L}_X$ will still preserve $\frakg$.

Let $\phi_s$  be the one-parameter family of diffeomorphisms generated by $X$, then  equation (\ref{eqn:soliton}) implies that
\[ g_t = (1-2\lambda t) \phi^*_{s(t)} (g) \]
is a solution to the Ricci flow.   Since  $\mathscr{L}_X$ preserves $\frakg$, it  defines a derivation of $\frakg$ which we call  $D$.  Since $X|_x = 0$, $\mathscr{L}_X$  also preserves  the subalgebra $\frakg_x $ as 
\[
[X, Y]|_x = \nabla_{X|_x} Y - \nabla_{Y|_x} X = 0.
\]
Therefore, $D$ will preserve $K$.     Defining $\Phi_s = \exp(sD)$ we  then see that  $\Phi(K) = K$ and 
\[ \phi_s(hK) = \Phi_s(h) K \qquad \text{for  } h \in G \]
showing that $(M,g)$ is a semi-algebraic Ricci soliton.  
\end{proof}
  
Proposition  \ref{prop:SASintrinsic} can also be thought of as giving an alternative, intrinsic definition of semi-algebraic Ricci solitons in terms of the sub-algebra $\frakg \subset \mathfrak{X}(M)$.  This leads to the following definition for non-homogeneous Ricci solitons.

\begin{definition} Let $(M,g,X)$ be a (not necessarily homogeneous) Ricci soliton and let $\frak{g}$ be an arbitrary subalgebra of $\mathfrak{X}(M)$.  Then $(M,g,X)$ is  \emph{ semi-algebraic with respect to $\frakg$} if  $\mathscr{L}_X: \mathfrak{X}(M) \rightarrow \mathfrak{X}(M)$ preserves $\frakg$.
\end{definition}
Proposition \ref{prop:SASintrinsic} shows this definition agrees with the definition of Jablonski in the homogeneous case.

\begin{rem}
We note that our definition of a semi-algebraic Ricci soliton depends on the choice of the vector field $X$: the condition that $\mathscr{L}_X$ preserve $\frakg$ is not invariant under adding a general Killing vector field to $X$. 
\end{rem}

Jablonski  also shows  in \cite{Jablonski} that every homogeneous Ricci soliton is algebraic with respect to its isometry group.  This result can easily be seen to generalize to the nonhomogeneous case. 

\begin{prop} 
Every Ricci soliton is semi-algebraic with respect to the subalgebra of Killing vector fields $\frakg = \mathfrak{iso}(M,g)$. 
\end{prop}

\begin{proof}
 let $Y$ be a Killing vector field, then
\begin{eqnarray*}
\mathscr{L}_{[X, Y]} g = \mathscr{L}_X \mathscr{L}_Y g - \mathscr{L}_Y \mathscr{L}_X g= -2 \mathscr{L}_Y\left(\Ric - \lambda g\right) = 0,
\end{eqnarray*}
i.e., $[X, Y]$ is also a Killing vector field.
\end{proof}


\medskip{}


\begin{thebibliography}{XXX9}

\bibitem[Al]{Alekseevskii} D. V. Alekseevski\v{i}, \emph{Classification of quaternionic spaces with transitive solvable group of motions}, (Russian) Izv. Akad. Nauk SSSR Ser. Mat. \textbf{39}(1975), no. 2, 315--362, 472.

\bibitem[AK]{AlekKim} D. V. Alekseevski\v{i} and B. N. Kimel'fel'd, \emph{Structure of homogeneous Riemannian spaces with zero Ricci curvature},(Russian) Funkcional. Anal. i Prilo\v{Z}en. \textbf{9}(1975), no. 2, 5--11.

\bibitem[Be]{Besse} A. L. Besse, \emph{Einstein manifolds}, Ergebnisse der Mathematik und ihrer Grenzgebiete(3){[}Results in Mathematics and Related Areas(3){]}, \textbf{10}. Springer-Verlag, Berlin, 1987.

\bibitem[CE]{CheegerEbin} J. Cheeger and D. G. Ebin, \emph{Comparison theorems in Riemannian geometry}, North-Holland Mathematical Library, Vol. 9. North-Holland Publishing Co., Amsterdam-Oxford; American Elsevier Publishing Co., Inc., New York, 1975.

\bibitem[Ha]{Hamilton} R. Hamilton. \emph{Three-manifolds with positive Ricci curvature}, J. Diff. Geom. \textbf{17}(1982), 255-306.

\bibitem[HPW1]{HPWLcf} C. He, P. Petersen and W. Wylie, \emph{On the classification of warped product Einstein metrics}, Comm. Anal. Geom., \textbf{20}(2012), No. 2, 271--312.

\bibitem[HPW2]{HPWconstantscal} C. He, P. Petersen and W. Wylie, \emph{Warped product Einstein metrics over spaces with constant scalar curvature}, accepted by Asian J. Math., arXiv: 1012.3446v1.

\bibitem[HPW3]{HPWuniqueness} C. He, P. Petersen and W. Wylie, \emph{Uniqueness of warped product Einstein metrics and applications}, arXiv: 1110.2456v2.

\bibitem[He]{Heber} J. Heber, \emph{Noncompact homogeneous Einstein spaces}, Invent. Math. \textbf{133}(1998), 279--352.

\bibitem[Ja]{Jablonski} M. Jablonski, \emph{Homogeneous Ricci solitons}, arXiv: 1109.6556v1, 2011.

\bibitem[Iv]{Ivey} T. Ivey. \emph{\ Ricci solitons on compact three-manifolds}, Differential Geom. Appl. \textbf{3}(1993), no. 4, 301--307.

\bibitem[KK]{KimKim} D.-S. Kim and Y.-H. Kim, \emph{Compact Einstein warped product spaces with nonpositive scalar curvature}, Proc. Amer. Math. Soc. \textbf{131}(2003), no. 8, 2573--2576.

\bibitem[Kn]{Knapp} A. Knapp, \emph{Lie groups beyond an introduction}, Second edition. Progress in Mathematics, \textbf{140}, Birkh\"{a}user Boston, Inc., Boston, MA, 2002.

\bibitem[LL]{LafuenteLauret} R. Lafuente and J. Lauret, \emph{Structure of homogeneous Ricci solitons and the Alekseevskii conjecture}, arXiv: 1212.6511v1. 

\bibitem[La1]{LauretNil} J. Lauret, \emph{Ricci soliton homogeneous nilmanifolds}, Math. Ann. \textbf{319}(2001), 715-733.

\bibitem[La2]{LauretSol} J. Lauret, \emph{Ricci soliton solvmanifolds}, J. Reine Angew. Math. \textbf{650}(2011), 1--21.

\bibitem[La3]{LauretStandard} J. Lauret, \emph{Einstein solvmanifolds are standard}, Ann. of Math. (2) 172 (2010), no. 3, 1859--1877.

\bibitem[Na]{Naber} A. Naber, \emph{Noncompact shrinking four solitons with nonnegative curvature}, J. Reine Angew. Math. \textbf{645}(2010), 125--153.

\bibitem[Per]{Perelman} G. Perelman. \emph{The entropy formula for the Ricci flow and its geometric applications}, arXiv: math.DG/0211159.

\bibitem[Pe]{PetersenGTM} P. Petersen, \emph{Riemannian geometry}, Second edition. Graduate Texts in Mathematics, \textbf{171}. Springer, New York, 2006.

\bibitem[PW1]{PWsymmetry} P. Petersen and W. Wylie, \emph{On gradient Ricci solitons with symmetry}, Proc. Amer. Math. Soc. \textbf{137}(2009), no. 6, 2085--2092.

\bibitem[PW2]{PWrigidity} P. Petersen and W. Wylie, \emph{Rigidity of gradient Ricci solitons}, Pacific J. Math. \textbf{241}(2009), no. 2, 329--345.

\bibitem[PW3]{PWclassification} P. Petersen and W. Wylie, \emph{On the classification of gradient Ricci solitons}, Geom. Topol. \textbf{14}(2010), no. 4, 2277--2300.

\bibitem[WZ]{WangZiller} M. Y. Wang and W. Ziller, \emph{Existence and nonexistence of homogeneous Einstein metrics}, Invent. Math. \textbf{84}(1986), no. 1, 177--194.

\bibitem[Wy]{Wylie} W. Wylie, \emph{Complete shrinking Ricci solitons have finite fundamental group}, Proc. Amer. Math. Soc. \textbf{136}(2008), 1803--1806.

\end{thebibliography}
\end{document}